\documentclass[preprint]{imsart}
\usepackage[reqno]{amsmath}
\usepackage{ogonek}
\usepackage{amsthm,amsfonts,amssymb}
\usepackage[usenames,dvipsnames]{xcolor}
\usepackage{graphicx}
\usepackage{subfigure}
\usepackage{url}
\RequirePackage{natbib}
\RequirePackage[colorlinks,citecolor=blue,urlcolor=blue]{hyperref}
\def\R{\mathbb R}
\def\Cov{{\rm {cov}}}
\def\Z{\mathbb Z}
\def \E{{\rm {E}}}

\def \P{\Pr}
\def \mb{\mathbf}
\newcommand{\bb}{\boldsymbol}
\newtheorem{theorem}{Theorem}[section]
\newtheorem{lemma}[theorem]{Lemma}

\newtheorem{proposition}[theorem]{Proposition}
\theoremstyle{definition}

\theoremstyle{remark}
\newtheorem{remark}[theorem]{Remark}
\allowdisplaybreaks
\numberwithin{equation}{section}

\begin{document}
\begin{frontmatter}

\title{Joint Asymptotics for Estimating the Fractal Indices of Bivariate
Gaussian Processes\thanks{Research supported in part by NSF grants DMS-1612885 and DMS-1607089.}}

\runtitle{Joint Asymptotics for Bivariate Fractal Indices}


\author{\fnms{Yuzhen} \snm{Zhou}\thanksref{a}\ead[label=e1]{yuzhenzhou@unl.edu}}
\and
\author{\fnms{Yimin} \snm{Xiao}\thanksref{b}\ead[label=e2] {xiao@stt.msu.edu}}
\address[a]{University of Nebraska-Lincoln \printead{e1}}
\address[b]{Michigan State University \printead{e2}}

\runauthor{Y. Zhou and Y. Xiao}
\begin{abstract}
Multivariate (or vector-valued) processes  are important for modeling 
multiple variables. The fractal indices of the components of the underlying 
multivariate process play a key role in characterizing the dependence 
structures and statistical properties of the multivariate process. 

In this paper, under the infill asymptotics framework, we establish
joint asymptotic results for the increment-based estimators
of bivariate fractal indices. Our main results  quantitatively describe
the effect of the cross-dependence structure on the performance 
of the estimators.
\end{abstract}

\end{frontmatter}

\section{Introduction}
\label{Introduction}

The fractal index 
of a stochastic process is useful for measuring the roughness of its
sample paths (e.g., it determines the Hausdorff dimension of the trajectories
of the process), and it is an important parameter in geostatistical models.
The problem of estimating the fractal index of a real-valued Gaussian or
non-Gaussian process has attracted the attention of many authors in past decades. Hall and Wood \cite{Hall_Wood_1993} studied the asymptotic properties of
the box-counting estimator of the fractal index. Constantine and Hall 
\cite{Constantine_Hall_1994} constructed estimators of the effective fractal dimension 
based on the variogram. Kent and Wood \cite{Kent_Wood_1997}
developed increment-based estimators for stationary
Gaussian processes on $\R$, which can achieve improved performance under infill
asymptotics (namely,  asymptotic properties of statistical procedures as the sampling
points grow dense in a fixed domain, see, e.g., \cite{Chen_Simpson_Ying_2000, Cressie_1993}).
Chan and Wood \cite{Chan_Wood_2000, Chan_Wood_2004} extended the method to a class 
of stationary Gaussian random fields defined on $\R^2$ and their transformations, which are 
non-Gaussian in general.  Zhu and Stein \cite{Zhu_Stein_2002}
expanded the work of Chan and Wood \cite{Chan_Wood_2000} by considering the fractional
Brownian surface. More recently, Coeurjolly \cite{Coeurjolly_2008} introduced a new
class of consistent estimators of the fractal dimension of locally self-similar
Gaussian processes on $\mathbb{R}$ using sample quantiles and derived 
the almost sure convergence and asymptotic normality for these estimators.
Bardet and Surgailis \cite{Bardet_Surgailis_2011} provided estimators of the fractal index based
on increment ratios for several classes of real-valued processes with rough
sample paths, including Gaussian processes, and studied their asymptotic properties.
Loh \cite{Loh_2015} constructed estimators from irregularly spaced data on
$\mathbb{R}^d$ with $d=1$ or $2$ via higher-order quadratic variations.
We refer to \cite{Gneiting_Sevcikova_Percival2012} and the references
therein for further information on various types of estimators and their
assessments.

In recent years, multivariate (or vector-valued) Gaussian processes and random fields
have become popular in modeling multivariate spatial datasets
(see, e.g., \cite{Gelfand_Diggle_Fuentes_Guttorp_2010, Wackernagel_2003}).
Several classes of multivariate spatial models were introduced in
\cite{Apanasovich_Genton_Sun_2012, Daley_Porcu_Bevilacqua_2015,  Du_Ma_2013,  Gneiting_Kleiber_Schlather2010,  Kleiber_Nychka_2012, Moreva_Martin_2016, Porcu_Daley_Buhmann_2013}. Two of the challenges in multivariate
modeling are to specify the cross-dependence structures and to quantify the
effect of the cross-dependence on the estimation and prediction performance.
We refer to \cite{Genton_2014_cross}  for an excellent review of the recent
developments in multivariate covariance functions. They also raised many
open questions and called for theoretical development of estimation and
prediction methodology in the multivariate context. To the best of our knowledge, 
only a few authors have worked in this direction; see, for example, \cite{Furrer_Bachoc_Du_2016, lim2008properties, pascual2006estimation, Ruiz_Porcu_2015, zhang2015doesn}. While with respect to focusing on estimating the fractal indices 
of a multivariate Gaussian process, we are only aware of the work by  Amblard
and Coeurjolly \cite{Amblard_Coeurjolly_2011}, in which they constructed estimators for the fractal indices 
of a class of multivariate fractional Brownian motions using discrete filtering techniques and studied 
their joint asymptotic distribution. By estimating the fractal index of each component separately, they found that 
the quality of these estimates was almost independent of the cross correlation of the multivariate fractional 
Brownian motion.

In this work, we consider a class of bivariate stationary Gaussian processes 
$\mathbf{X} \triangleq \{(X_1(t),X_2(t))^\top, t\in \mathbb{R}\}$ (the operator 
$(\cdot)^\top$ means the transpose of a vector or a matrix)  and study the 
joint asymptotic properties of the estimators for the fractal indices of the components
$X_1$ and $X_2$ under the infill asymptotics framework.  Our main purpose 
is to clarify the effect of cross covariance on the performance of the joint estimators.

More specifically, we assume that  $\mathbf{X} $ has mean
${\rm E}\mathbf{X}(t)=\mb 0$ and matrix-valued covariance function
\begin{equation}\label{Cov}
\mb C(t)=\left(\begin{array}{ll}
C_{11}(t)&C_{12}(t)\\
C_{21}(t)&C_{22}(t)
\end{array}
\right),
\end{equation}
where $C_{ij}(t):={\rm E}[X_i(s)X_j(s+t)],\ i=1,2.$
Further, we assume that the following conditions are satisfied
\begin{align}
\label{covariance assumption}
\begin{split}
C_{11}(t)&=\sigma_1^2-c_{11}|t|^{\alpha_{11}}+o(|t|^{\alpha_{11}}),\\
C_{22}(t)&=\sigma_2^2-c_{22}|t|^{\alpha_{22}}+o(|t|^{\alpha_{22}}),\\
C_{12}(t)&=C_{21}(t)=\rho \sigma_1\sigma_2(1-c_{12}
|t|^{\alpha_{12}}+o(|t|^{\alpha_{12}})),
\end{split}
\end{align}
where $\alpha_{11},\alpha_{22}\in (0,2)$, $\sigma_1,\sigma_2>0$,
$|\rho|\in (0,1)$ and $c_{11},\, c_{22},c_{12}>0$ are constants. Under the
assumption (\ref{covariance assumption}),  in order for (\ref{Cov})
to be a valid covariance function, it is necessary to impose some
restrictions on the parameters $(\alpha_{11},\alpha_{22}, \alpha_{12})$.
In this paper, we assume
\begin{equation}\label{Eq:valid}
\begin{split}
&\frac{\alpha_{11}+\alpha_{22}}{2} < \alpha_{12}, \ \ \hbox{ or } \\
&\frac{\alpha_{11}+\alpha_{22}}{2} = \alpha_{12} \ \ \hbox{ and }\  \
c_{12}^2 \rho^2 \sigma_1^2 \sigma_2^2 < c_{11} c_{22}.
\end{split}
\end{equation}
This is a mild assumption. See Appendix A for justification.

Henceforth, we refer to (\ref{covariance assumption}) and (\ref{Eq:valid})
as Condition (${\mb A1}$). Under the assumption \eqref{covariance assumption},
it is well known (see, e.g., \cite{Adler_1981}, Theorem $8.4.1$) that
the fractal dimensions of the trajectories of each component $X_1$ and
$X_2$ are given by
\[
\dim {\rm Gr} X_1 ([0, 1]) = 2- \frac{\alpha_{11}} 2, \quad \hbox{ a.s.}
\]
and
\[
\dim {\rm Gr} X_2 ([0, 1]) = 2- \frac{\alpha_{22}} 2, \quad \hbox{ a.s.},
\]
respectively. Above, for $ i \in\{ 1, 2\}$, ${\rm Gr} X_i([0, 1]) =
\{(t, X_i(t)): t \in [0, 1]\}$ is the trajectory (or graph set)
of the real-valued process $X_i = \{X_i(t), t \in \mathbb{R}\}$ over the
interval $[0, 1]$. A bivariate stationary Gaussian process $\mathbf{X}$ with 
matrix-valued covariance
function \eqref{Cov} that satisfies Condition (${\mb A1}$)
has richer fractal properties. For example, we consider  the trajectory  of
$\mathbf{X}$ on $[0, 1]$, which is ${\rm Gr}\mathbf{X}([0,1]) = \{(t, X_1(t),\,
X_2(t))^\top: t \in [0, 1]\} \subseteq \mathbb R^3$.  For notational convenience, we further assume
that $\alpha_{11} \le \alpha_{22}$
(otherwise we may relabel the components of $\mathbf{X}$). Then, we can
apply Theorem 2.1 in \cite{Xiao_1995} to show that, with probability 1,
\begin{equation}\label{Xiao95}
\begin{split}
\dim {\rm Gr} \mathbf{X}([0, 1]) &= {\rm {min}}\bigg\{\frac{2 + \alpha_{22}
- \alpha_{11}} {\alpha_{22}} , \,3 - \frac{\alpha_{11} + \alpha_{22}} 2\bigg\}\\
&= \left\{\begin{array}{ll}
\frac{2 + \alpha_{22} - \alpha_{11}} {\alpha_{22}}, \ \
&\hbox{ if }\, \alpha_{11} + \alpha_{22}\ge 2,\\
3 - \frac{\alpha_{11} + \alpha_{22}} 2, \ \ &\hbox{ if }\,
\alpha_{11} + \alpha_{22}< 2.
\end{array}
\right.
\end{split}
\end{equation}
This result shows that the indices $\alpha_{11}$ and $\alpha_{22}$
determine the fractal dimension of the trajectory of the bivariate
Gaussian process $\mathbf{X}$. Furthermore, one can characterize
many other fractal properties of $\mathbf{X}$ explicitly in terms
of these indices. See \cite{xiao2013recent} for a recent overview. Hence,
analogous to the univariate case, it is natural to call $(\alpha_{11},
\alpha_{22})$ the fractal indices of $\mathbf{X}$. For the reader's convenience,
we include a proof of \eqref{Xiao95} in Appendix B.

Although the parameters $\alpha_{11}$ and $\alpha_{22}$ can
be estimated separately from observations of the coordinate
processes $X_1$ and $X_2$, \eqref{Xiao95} suggests that, in doing so,
one might miss some important information about the structures of the
bivariate process $\mathbf{X}$. For example, although the estimator of
$\dim {\rm Gr} \mathbf{X}([0, 1])$ can be obtained by plugging the
estimators of $\alpha_{11}$ and $\alpha_{22}$ into \eqref{Xiao95}, say
$(\hat \alpha_{11}, \hat \alpha_{22})$, we cannot evaluate the estimation
efficiency without  the joint asymptotic properties of $(\hat \alpha_{11},
\hat \alpha_{22})$.  Hence, it is necessary to study these estimators jointly
and to quantify the effect of the cross-covariance on their performance.

In this paper, we consider the increment-based estimators of
$\alpha_{11}$ and $\alpha_{22}$, denoted by $\hat \alpha_{11}$ and
$\hat \alpha_{22}$, respectively, and study the bias, mean square
error matrix and joint asymptotic distribution under the infill
asymptotics framework. The main results are given in Theorems \ref{thm: Bias}
$\sim$ \ref{thm: Joint normality of alpha_11_22}. In particular,
we prove that $\sqrt{n}\hat \alpha_{11}$ and $\sqrt{n}\hat \alpha_{22}$
are asymptotically uncorrelated if $(\alpha_{11}
+\alpha_{22})/2<\alpha_{12}$, while they are asymptotically correlated if
$(\alpha_{11}+\alpha_{22})/2=\alpha_{12}$. Our results are
applicable to a wide class of bivariate Gaussian processes,  including the bivariate  
Mat\'{e}rn model introduced by Gneiting, Kleiber and Schlather \cite{Gneiting_Kleiber_Schlather2010}, 
the bivariate powered exponential model and  bivariate Cauchy model of Moreva and Schlather
 \cite{Moreva_Martin_2016}, and a class of bivariate models introduced by Du and Ma 
 \cite{Du_Ma_2013}.

This paper raises several open questions. First, the method of joint asymptotics developed 
in this paper and the recent work by Loh \cite{Loh_2015} on constructing estimators for the
univariate fractal index given irregularly spaced data make it possible to study the joint 
asymptotics in estimating bivariate fractal indices when data are observed irregularly 
on $\mathbb{R}^2$. This problem is interesting from both theoretical and application viewpoints, 
but it appears to be challenging. Second, the work by Ruiz-Medina and Porcu 
\cite{Ruiz_Porcu_2015}, which established conditions for the equivalence of Gaussian measures 
of multivariate random fields, makes it promising to generalize the consistency and asymptotic 
normality results of maximum likelihood estimators for a univariate random field to the case of 
multivariate Gaussian fields. The existing results in the univariate case were established under 
the assumption that the smoothness parameter is known (see, e.g., 
\cite {Du_Zhang_Mandrekar_2009, Kaufman_Schervish_Nychka_2008, Zhang_2004}). It 
would be interesting to study whether the asymptotic properties hold  in either the univariate or multivariate 
case while plugging in estimators of the smooth parameters.

The rest of this paper is organized as follows. We follow Chan and Wood 
\cite{Chan_Wood_2000} and Kent and Wood \cite{Kent_Wood_1997} and 
formulate the increment-based estimators for
$(\alpha_{11}, \alpha_{22})$ in Section \ref{sec: the increment-based estimators}.
In Section \ref{Asymp-prop}, we state the main results of the joint
asymptotics of the bivariate estimators. An application to the non-smooth
bivariate Mat\'{e}rn processes is given in Section
\ref{sec: an example_bivariate Matern field}. In Section \ref{sec: simulation},
we present a simulation study on the efficiency of the estimators. The
proofs of our main results are given in Section \ref{sec_proof of main results}.
Finally, some auxiliary results and their proofs are included in the Appendix.

We end the introduction with some notation. $\Z^+$ denotes the
set of all positive integers, and $\mathcal{B}(\R)$ is the collection
of all Borel sets on $\R$. For any real-valued sequences $\{a_n\}_{n=1}^\infty$,
$\{b_n\}_{n=1}^\infty$, $a_n \sim b_n$ means $\lim_{n\rightarrow \infty}
{b_n}/{a_n}=1$,  $a_n \gtrsim b_n$ means that there exists a constant $c>0$
such that  $a_n \geq c\, b_n$ for all $n$ sufficiently large and $a_n \asymp b_n$
means  $a_n \gtrsim b_n$ and  $b_n \gtrsim a_n$. Similar notation is 
used for functions of continuous variables.

An unspecified positive and finite constant will
be denoted by $C_0$. More specific constants are numbered as $C_1,
\, C_2, \ldots.$

\section{The increment-based estimators}
\label{sec: the increment-based estimators}
Assume that the values of the bivariate process  $\mb X$ are observed regularly on an  interval $I$, say $I =[0,1]$.
More specifically, we have $n$ pairs of observations $(\mb X(1/n), \ldots, \mb X(1))^\top$.
By applying the  increment-based method introduced by Kent and Wood      \cite{Kent_Wood_1997}
for estimating the fractal index of a real-valued locally self-similar
Gaussian process (see also \cite{Chan_Wood_2000, Chan_Wood_2004} for
further development),
we can estimate the fractal indices $(\alpha_{11}, \alpha_{22})$ of $\mb X$.
Our emphasis is on studying the joint asymptotic properties of the estimators.
In particular, we study the effect of cross-covariance on their joint performance.

Let $m \ge 2$ be a fixed integer. For each component $X_i,\ i=1,2$ and  
integer $u\in \{1,\ldots, m\}$, we define the dilated filtered discretized process
with second difference (see, e.g., \cite{Kent_Wood_1997}),
\begin{align*}
&Y_{n,i}^u(j):=n^{\frac{\alpha_{ii}}{2}}\left(X_i\Big(\frac{j-u}{n}\Big)-2X_i
\Big(\frac{j}{n}\Big)+X_i\Big(\frac{j+u}{n}\Big)\right),\ j=1,\ldots,n.
\end{align*}
Denote by $a_{-1}=1,a_0=-2,a_1=1$. $Y_{n,i}^u$ can be rewritten as
\begin{align*}
&Y_{n,i}^u(j)=n^{\frac{\alpha_{ii}}{2}}\sum_{k=-1}^1 a_kX_i\bigg(\frac{j+ku}{n}\bigg).
\end{align*}
As in Kent and Wood \cite{Kent_Wood_1997}, one can verify that, under 
(\ref{covariance assumption}), $Y_{n,i}^u(j)$ is a Gaussian random 
variable with mean 0, and its variance converges to 
$c_{ii}(8 - 2^{\alpha_{ii}+1})u^{\alpha_{ii}}$
(this follows from (\ref{sigma_n,ii^uv(h)2}) below).
Let $Z_{n,i}^u(j):=(Y_{n,i}^u(j))^2$ and  define
\begin{align} \label{Def:Z}
\bar Z_{n,i}^u:=\frac{1}{n}\sum_{j=1}^n Z_{n,i}^u(j).
\end{align}
For $i= 1, 2$, it follows from \cite{Kent_Wood_1997} that, under certain
regularity conditions on the covariance function $C_{ii}(t)$, we have
\begin{align*}
\bar Z_{n,i}^u\xrightarrow{p} C_i u^{\alpha_{ii}},
\end{align*}
where $\xrightarrow{p}$ represents convergence in probability and
$C_i = c_{ii}(8 - 2^{\alpha_{ii}+1})$. Hence,
\begin{align*}
\ln \bar Z_{n,i}^u\xrightarrow{p}\alpha_{ii}\ln u+\ln C_i,\  \ i=1, 2,
\end{align*}
where $\ln$ represents natural logarithm. Consequently, the fractal indices $\alpha_{ii}$  ($i = 1, 2$) can be estimated
by linear regression of $\ln \bar Z_{n,i}^u$ on $\ln u$ for $u = 1, \ldots, m$.

In this paper,  we employ Chan and Wood \cite{Chan_Wood_2000}'s linear 
estimators for $\alpha_{ii}$ based on $\ln \bar Z_{n,i}^u$, that is,
\begin{align} \label{Def:alpha}
\hat\alpha_{ii}=\sum_{u=1}^m L_{u,i}\ln \bar Z_{n,i}^u,
\end{align}
where $\{L_{u,i}, u=1,\ldots,m\} \, (i=1,2)$ are finite sequences  
of real numbers such that
\begin{align}
\label{L_u,i assumption}
\sum_{u=1}^mL_{u,i}=0\ \ \ \ \text{ and }\ \ \  \sum_{u=1}^mL_{u,i}\ln u=1.
\end{align}
Both the ordinary least squares and generalized least squares
estimators introduced by Kend and Wood \cite{Kent_Wood_1997} are examples of
the above estimators. We remark that due to the first condition in (\ref{L_u,i assumption}), 
the estimators $\hat\alpha_{ii}$ ($i = 1, 2$) defined in (\ref{Def:alpha}) can be computed 
from the observed values $(\mb X(1/n), \ldots, \mb X(1))^\top$ and do not depend on 
the unknown indices $\alpha_{ii}$. 

\section{Joint asymptotic properties}
\label{Asymp-prop}
For $i = 1, 2$, let
\begin{align*}
\mb {\bar  Z}_{n,i}=(\bar Z_{n,i}^1,\ldots,
\bar Z_{n,i}^m)^\top
\end{align*}
and denote
\begin{align*}
\mb {\bar Z}_{n}=(\mb{\bar Z}_{n,1}^\top,\mb{\bar Z}_{n,2}^\top)^\top.
\end{align*}
Under the infill asymptotics framework, we first study the asymptotic
properties of $\mb{\bar Z}_n$ in Section
\ref{sec:variance of Zbar and asymptotic normality}. In Section
\ref{sec: asymptotic properties of alpha1 and alpha2}, the joint
asymptotic properties of the estimators $(\hat\alpha_{11},
\hat\alpha_{22})^\top$ are obtained.

\subsection{Variance of $\mb{\bar Z}_n$ and asymptotic normality}
\label{sec:variance of Zbar and asymptotic normality}
First, given $u,v=1,\ldots, m$, we consider the covariance matrix of
$(Y_{n,1}^u,\,Y_{n,2}^v)^\top$. For $i = 1, 2$, it follows from 
Kent and Wood \cite{Kent_Wood_1997} that the marginal covariance function for 
$Y_{n,i}^u$ and  $Y_{n,i}^v$ is
\begin{align}
\label{sigma_n,ii^uv(h)}
\sigma_{n,ii}^{uv}(h)&:={\rm E}[Y_{n,i}^u(\ell)Y_{n,i}^v(\ell+h)]\nonumber\\
&\rightarrow -c_{ii}\sum_{j,k=-1}^1a_ja_k|h+kv-ju|^{\alpha_{ii}}\triangleq
\sigma_{0,ii}^{uv}(h),
\end{align}
as $\ n\rightarrow \infty$. In particular, we derive that the variance
of $Y_{n,i}^u(\ell)$ satisfies
\begin{equation}\label{sigma_n,ii^uv(h)2}
\sigma_{n,ii}^{uu}(0) \rightarrow C_i\, u^{\alpha_{ii}},
\ \ \ \ \ \hbox{ as } \, n \to \infty,
\end{equation}
where $C_i = c_{ii} (8 - 2^{\alpha_{ii}+1})$.

Under the assumption \eqref{covariance assumption}, the
cross covariance between $Y_{n,1}^u$ and $Y_{n,2}^v$
can be derived as follows.
\begin{equation}
\label{sigma_n,12^uv(h)}
\begin{split}
&\sigma_{n,12}^{uv}(h):={\rm E}[Y_{n,1}^u(\ell)Y_{n,2}^v(\ell+h)]\\
&=n^{\frac{\alpha_{11}+\alpha_{22}}{2}}\sum_{j,k=-1}^1a_ja_kC_{12}
\bigg(\frac{h+kv-ju}{n}\bigg) \\
&\rightarrow
\left\{\begin{array}{ll}
0,&\text{if}\ \frac{\alpha_{11}+\alpha_{22}}{2}<\alpha_{12},\\
-\rho \sigma_1\sigma_2 c_{12}\sum_{j,k=-1}^{1}a_ja_k|h+kv-ju|^{\alpha_{12}}
,& \text{if}\ \frac{\alpha_{11}+\alpha_{22}}{2}=\alpha_{12}
\end{array}
\right.\\
& \triangleq \sigma_{0,12}^{uv}(h).
\end{split}
\end{equation}
Therefore, if $(\alpha_{11}+\alpha_{22})/2<\alpha_{12}$, the
covariance matrix of $(Y_{n,1}^u(\ell),Y_{n,2}^v(\ell+h))^\top$ satisfies
\begin{align*}
\text{Var}\left(
\begin{array}{l}
Y_{n,1}^u(\ell)\\
Y_{n,2}^v(\ell+h)
\end{array}
\right)\rightarrow \left(
\begin{array}{ll}
\sigma_{0,11}^{uu}(0)&0\\
0&\sigma_{0,22}^{vv}(0)
\end{array}
\right),\ \text{as}\ n\rightarrow \infty.
\end{align*}
If $(\alpha_{11}+\alpha_{22})/2=\alpha_{12}$,
the covariance matrix of $(Y_{n,1}^u(\ell),Y_{n,2}^v(\ell+h))^\top$
satisfies
\begin{align*}
\text{Var}\left(
\begin{array}{l}
Y_{n,1}^u(\ell)\\
Y_{n,2}^v(\ell+h)
\end{array}
\right)\rightarrow \left(
\begin{array}{ll}
\sigma_{0,11}^{uu}(0)&\sigma_{0,12}^{uv}(h)\\
\sigma_{0,12}^{uv}(h)&\sigma_{0,22}^{vv}(0)
\end{array}
\right),\ \ \ \text{ as }\ n \rightarrow \infty.
\end{align*}
We adapt the method of derivation in Section $3$ of Kent and Wood
\cite{Kent_Wood_1997} to find the covariance matrix of the
random vector $\mb{\bar Z}_n$. Using the fact that
if  $(U,V)\sim \mathcal{N}\bigg(\Big(\begin{array}{l}0\\0\end{array}\Big),\ 
 \Big(\begin{array}{ll}1 & \xi\\ \xi & 1\end{array}\Big)\bigg)$
is a bivariate normal random vector, 
then $\Cov(U^2,V^2)=2\xi^2$, we obtain
\begin{align*}
\Cov(Z_{n,i}^u(\ell),Z_{n,j}^v(\ell+h))
&=2(\sigma^{uv}_{n,ij}(h))^2,\ \ \  i,j=1,2;
\end{align*}
hence,
\begin{align*}
\phi_{n,ij}^{uv}:=\Cov(\bar Z_{n,i}^u,\bar Z_{n,j}^v)
=\frac{1}{n}\sum_{h=-n+1}^{n-1}\bigg(1-\frac{|h|}{n}\bigg)\times
2(\sigma_{n,ij}^{uv}(h))^2.
\end{align*}
Denote by $\mb \Phi_{n,ij}=(\phi_{n,ij}^{uv})_{u,v=1}^m$
the covariance matrix of $\mb{\bar Z}_{n,i}$ and $\mb {\bar Z}_{n,j}$.
Then,  the covariance matrix of $\mb{\bar Z}_n$ can be written as
\begin{align*}
\mb \Phi_n=\left(\begin{array}{ll}
\mb\Phi_{n,11}&\mb\Phi_{n,12}\\
\mb\Phi_{n,21}&\mb\Phi_{n,22}
\end{array}
\right).
\end{align*}
To study the asymptotic properties of $\mb\Phi_n$
and $\mathbf{\bar Z}_n$, we impose an additional
regularity condition on the fourth derivative of the functions $C_{ij}(t)$ in 
(\ref{Cov}) around the origin, which is analogous to the condition ($A_4$) in \cite{Kent_Wood_1997} 
and will be called Condition (${\mb A2}$):
\begin{equation*}
\label{derivative of covariance}
\begin{split}
C_{11}^{(4)}(t)&=-\frac{c_{11}\alpha_{11}!}{(\alpha_{11}-4)!}
|t|^{\alpha_{11}-4}+o(|t|^{\alpha_{11}-4}), \\
C_{22}^{(4)}(t)&=-\frac{c_{22}\alpha_{22}!}{(\alpha_{22}-4)!}
|t|^{\alpha_{22}-4}+o(|t|^{\alpha_{22}-4}),\\
C_{12}^{(4)}(t)&=C_{21}^{(4)}(t)=-\rho \sigma_1\sigma_2
\frac{c_{12}\alpha_{12}!}{(\alpha_{12}-4)!}
|t|^{\alpha_{12}-4}+o(|t|^{\alpha_{12}-4}).
\end{split}
\end{equation*}
Above, for any $\alpha > 0$, $\alpha!/(\alpha-4)! = \alpha (\alpha-1)(\alpha-2)(\alpha-3)$.


For $i, j = 1, 2$, let $\phi_{0,ij}^{uv}=2\sum_{h=-\infty}^{\infty}
(\sigma_{0,ij}^{uv}(h))^2$,  which is convergent,  $\mb\Phi_{0,ij}
=(\phi_{0,ij}^{uv})_{u,v=1}^m$, and let
\begin{align*}
\mb\Phi_0=\left(\begin{array}{ll}
\mb\Phi_{0,11}&\mb\Phi_{0,12}\\
\mb\Phi_{0,21}&\mb\Phi_{0,22}
\end{array}
\right).
\end{align*}

The following theorems describe the asymptotic properties of the
random vector $\mb{\bar Z}_n$.
\begin{theorem}
\label{thm: asymptotic variance of Z_1&Z_2}
If Conditions (${\mb A1}$) and (${\mb A2}$) hold, then
\begin{align}
\label{nPhi_n convergence}
n\bb\Phi_n\rightarrow \bb\Phi_0, \ \text{as}\ n\rightarrow \infty.
\end{align}
Moreover, if  $(\alpha_{11}
+\alpha_{22})/{2}<\alpha_{12}$, then
$\bb\Phi_{0,12}=\bb\Phi_{0,21}=\mathbf{0}$.
\end{theorem}

\begin{theorem}
\label{thm: asymptotic normality of bar Z}
If Conditions (${\mb A1}$) and (${\mb A2}$) hold, then
\begin{align*}
n^{1/2}(\mb{\bar Z}_n-{\rm E}[\mb{\bar Z}_n])
\xrightarrow{d} \mathcal{N}_{2m}(\mb 0,\bb\Phi_0), \ \text{as}\
n\rightarrow \infty,
\end{align*}
where $\mathcal{N}_{2m}(\mb 0,\bb\Phi_0)$ is the $(2m)$-dimensional normal distribution
with mean $\mb 0$ and covariance matrix $\bb\Phi_0$.
\end{theorem}
\setcounter{theorem}{0}

\begin{remark}
Theorem \ref{thm: asymptotic normality of bar Z} extends Theorem $2$ in Kent and Wood 
\cite{Kent_Wood_1997} to the bivariate case, and shows that $\sqrt{n}\mb{\bar Z}_{n,1}$
and $\sqrt{n}\mb{\bar Z}_{n,2}$ are asymptotically independent when
$(\alpha_{11}+\alpha_{22})/2<\alpha_{12}$. The proofs of Theorems
\ref{thm: asymptotic variance of Z_1&Z_2} and \ref{thm: asymptotic normality of bar Z}
are given in Appendix D.
\end{remark}

\begin{remark}
The class of matrix-valued covariance functions whose properties around the origin 
satisfy ($\mb A1$) and ($\mb A2$) is large, including such significant examples as 
the bivariate  Mat\'{e}rn model of Gneiting, Kleiber and Schlather \cite{Gneiting_Kleiber_Schlather2010}, 
the bivariate powered exponential model and  bivariate Cauchy model of Moreva and Schlather 
\cite{Moreva_Martin_2016}, the bivariate Wendland-Gneiting covariance function of Daley, Procu and Bevilacqua \cite{Daley_Porcu_Bevilacqua_2015} and a class of bivariate models introduced by Du and Ma 
\cite{Du_Ma_2013}, such as Example 3. Since the matrix-valued covariance functions in these references have 
explicit closed forms,  Conditions ($\mb A1$) and ($\mb A2$) can be verified directly by using Taylor's expansion or 
L'Hospital's rule.
\end{remark}

\begin{remark}
Another way to verify Conditions ($\mb A1$) and ($\mb A2$) is to make use of the spectral 
representation of $C_{ij}$:
\[
C_{ij}(t) = \int_{\mathbb R} \cos (t \xi) F_{ij}(d \xi),
\]
where  $F_{ij}$ is the spectral measure of $C_{ij}$. Writing 
\[
C_{ij}(0) - C_{ij}(t) = \int_{\mathbb R} (1 - \cos (t \xi)) F_{ij}(d \xi),
\]
one can  see that Condition ($\mb A1$)  may follow from an Abelian-type 
theorem and the tail behavior of the spectral measure $F_{ij}$ at infinity (see, for example, \cite{pitman1968}).  

To verify Condition ($\mb A2 $),  we may assume that $F_{ij}$ has a
density function $f_{ij}(\xi)$ which decays faster than certain polynomial rate as  $|\xi| \to \infty$. 
A change of variable yields that for $t \ne 0$,
\[
C_{ij}(0) -  C_{ij}(t) = \frac 1 t \int_{\mathbb R}\big(1 -  \cos \xi\big)\, f_{ij}\big(\frac{\xi}{t}\big) \, d\xi.
\]
Then we can differentiate $C_{ij}(t)$ as follows:
\[
\begin{split}
C_{ij}'(t) &=  \frac 1 {t^2} \int_{\mathbb R} \big(1 - \cos \xi \big)\, f_{ij}\big(\frac{\xi}{t}\big) \, d\xi + \frac 1 {t}  
\int_{\mathbb R} \big(1 - \cos \xi \big)\, f_{ij}'\big(\frac{\xi}{t}\big) \, \frac \xi {t^2}\, d\xi\\
&=  \frac {C_{ij}(0)- C_{ij}(t)} t + \frac 1 {t}  
\int_{\mathbb R} \big(1 - \cos \xi \big)\, f_{ij}'\big(\frac{\xi}{t}\big) \, \frac \xi {t^2}\, d\xi.
\end{split}
\]
Consequently, the asymptotic behavior of $C_{ij}'(t) $ as $t \to 0$ can be derived from ($\mb A1$) 
and another application of the Abelian-type theorem in \cite{pitman1968} to the second integral. 
Iterating this procedure  three more times, we can verify Condition ($\mb A2 $).
\end{remark}

\setcounter{theorem}{2}

\subsection{Asymptotic properties of $(\hat \alpha_{11},\,
\hat\alpha_{22})^\top$}
\label{sec: asymptotic properties of alpha1 and alpha2}
This section contains the main results of this paper.
We make a stronger assumption by specifying the remainder
terms in Assumption \eqref{covariance assumption}.
Suppose that for some constants $\beta_{11},\,\beta_{22},\,\beta_{12}>0$,
\begin{equation}
\label{covariance assumption 2}
\begin{split}
C_{11}(t)&=\sigma_1^2-c_{11}|t|^{\alpha_{11}}+O(|t|^{\alpha_{11}+\beta_{11}}), \\
C_{22}(t)&=\sigma_2^2-c_{22}|t|^{\alpha_{22}}+O(|t|^{\alpha_{22}+\beta_{22}}), \\
C_{12}(t)&=C_{21}(t)=\rho \sigma_1\sigma_2(1-c_{12}|t|^{\alpha_{12}}+O(|t|^{\alpha_{12}+\beta_{12}})).
\end{split}
\end{equation}
We label the three conditions in (\ref{covariance assumption 2}),
together with (\ref{Eq:valid}), as Condition (${\mb A3}$).

Let $\hat {\bb \alpha} =(\hat \alpha_{11},\, \hat \alpha_{22})^\top$
be the estimators of the fractal indices $\bb\alpha =(\alpha_{11},
\alpha_{22})^\top$, as defined in (\ref{Def:alpha}). The theorems
below establish the asymptotic properties of $\hat {\bb \alpha}$,  including  the bias, mean
square error matrix and their joint asymptotic distribution.
\begin{theorem}[Bias]
\label{thm: Bias}
Assume Conditions ($\mb{A2}$) and ({$\mb A3$}) hold. Then, for
the estimators $\hat \alpha_{ii}, i=1,2$, we have
\begin{align*}
{\rm E} \big[ \hat \alpha_{ii}-\alpha_{ii} \big]=O(n^{-1})
+O(n^{-\beta_{ii}}),\ i=1,2.
\end{align*}
\end{theorem}

\begin{theorem}[Mean square error matrix]
\label{thm: Mean Suare Error Matrix}
Assume ($\mb{A2}$) and ({$\mb A3$}) hold.
If $(\alpha_{11}+\alpha_{22})/2=\alpha_{12}$, then
\begin{align} \label{Eq:16}
&{\rm E}[(\bb{\hat\alpha}-\bb\alpha)(\bb{\hat \alpha}-\bb\alpha)^\top]
\nonumber\\
&=\left(
\begin{array}{ll}
O(n^{-1}))&O(n^{-1})\\
O(n^{-1})&O(n^{-1})
\end{array}
\right)+\left(
\begin{array}{ll}
O(n^{-\psi(\beta_{11},\beta_{11})})&O(n^{-\psi(\beta_{11},\beta_{22})})
\\
O(n^{-\psi(\beta_{11},\beta_{22})})
&O(n^{-\psi(\beta_{22},\beta_{22})})
\end{array}
\right).
\end{align}
Here and below,  $\psi(x_1,x_2):={\rm {min}}\{1+x_1,1+x_2,x_1+x_2\}$.

If $(\alpha_{11}+\alpha_{22})/2<\alpha_{12}$, then
\begin{align} \label{Eq:17}
&{\rm E}[(\hat{\bb\alpha}-\bb\alpha)(\hat{\bb \alpha}-\bb\alpha)^\top]
\nonumber\\
&=\left(
\begin{array}{ll}
O(n^{-1})&o(n^{-1})\\
o(n^{-1})&O(n^{-1})
\end{array}
\right)+\left(
\begin{array}{ll}
O(n^{-\psi(\beta_{11},\beta_{11})})&O(n^{-\psi(\beta_{11},\beta_{22})})
\\
O(n^{-\psi(\beta_{11},\beta_{22})})
&O(n^{-\psi(\beta_{22},\beta_{22})})
\end{array}
\right).
\end{align}
\end{theorem}
\setcounter{theorem}{3}

\begin{remark}
The constants $\beta_{11}, \beta_{22}$ and $\beta_{12}$ from
(\ref{covariance assumption 2}) appear in both the bias and mean
square error matrix (Theorems \ref{thm: Bias} and
\ref{thm: Mean Suare Error Matrix}) because the remainder
terms $O(|t|^{\alpha_{ii}+\beta_{ii}})$ in the covariance
function are ignored in the estimation procedure,
which might strongly affect the efficiency of the estimators
(see, e.g., \cite{Kent_Wood_1997}). The statistical performance
of the estimators $(\hat \alpha_{11},\, \hat \alpha_{22})^\top$
can be significantly improved if more detailed information
on the remainder term is available. In Section
\ref{sec: an example_bivariate Matern field}, we show
that this is indeed the case when $\mathbf{X}$ is a nonsmooth bivariate Mat\'{e}rn process.
\end{remark}
\setcounter{theorem}{4}

Finally, we study the asymptotic distribution of $\hat {\bb\alpha}$
by applying multivariate delta methods (see, e.g.,
\cite{Amblard_Coeurjolly_2011, Lehmann_2006}). By \eqref{Def:Z},
\eqref{sigma_n,ii^uv(h)}, and \eqref{covariance assumption 2}, we have
\begin{equation}\label{Eq:Ztau}
{\rm E}\bar Z^u_{n,i}= {\rm E}[(Y_{n,i}^u(0))^2]= \tau_{u,i}
(1+O(n^{-\beta_{ii}})),
\end{equation}
where $\tau_{u,i}=c_{ii} (8 - 2^{\alpha_{ii}+1})u^{\alpha_{ii}}$.  Let
$$
\widetilde{\mathbf L}_i=(L_{1,i}/\tau_{1,i},\ldots,
L_{m,i}/\tau_{m,i})^\top,
\ \ \ i=1,2.
$$

The following theorem provides the joint asymptotic distribution of
$\hat {\bb\alpha}$.
\begin{theorem}[Asymptotic distribution]
\label{thm: Joint normality of alpha_11_22}
Assume ($\mb{A2}$) and ({$\mb A3$}) hold with $\beta_{11},
\beta_{22}>1/2$. Then, $\sqrt{n} (\hat {\bb\alpha}-\bb\alpha)$
follows the asymptotic properties below.
\begin{align*}
\sqrt{n}\left(
\begin{array}{l}
\hat\alpha_{11}-\alpha_{11}\\
\hat \alpha_{22}-\alpha_{22}
\end{array}
\right)\xrightarrow d \mathcal{N}\left(\left(\begin{array}{l}
0\\
0
\end{array}
\right),
\left(
\begin{array}{ll}
\widetilde {\mathbf L}_1^\top \mb\Phi_{0,11}
\widetilde{\mathbf L}_1&\widetilde {\mb L}_1^\top
\mb\Phi_{0,12}\widetilde {\mathbf L}_2\\
\widetilde {\mathbf L}_2^\top \mb\Phi_{0,21}\widetilde{\mathbf L}_1
&\widetilde {\mathbf L}_2^\top
 \mb\Phi_{0,22}\widetilde {\mathbf L}_2
\end{array}
\right)
\right).
\end{align*}
Specifically, if  $(\alpha_{11}+\alpha_{22})/{2}<\alpha_{12}$, then $\sqrt n \hat \alpha_{11}$
and $\sqrt n\hat \alpha_{22}$ are asymptotically independent.
\end{theorem}
\setcounter{theorem}{4}

\begin{remark}
The current estimation procedure and asymptotic properties are derived for nonsmooth bivariate Gaussian models, 
that is, the smoothness parameters $\alpha_{ii} \in (0,2)$ for $i=1,2$. If the sample function of the 
component $X_i$  is almost surely differentiable, then the corresponding index $\alpha_{ii} \ge 2$ in 
\eqref{covariance assumption}. 
In this case, one may extend the idea of Kent and Wood \cite{Kent_Wood_1995} and consider the covariance 
functions with the following local properties
\begin{align*}
\label{covariance assumption-extend}
\begin{split}
C_{11}(t)&=\sigma_1^2-\sum_{k=1}^{q} b_{1,k} t^{2j}-c_{11}|t|^{\alpha_{11}}+o(|t|^{\alpha_{11}}),\\
C_{22}(t)&=\sigma_2^2-\sum_{k=1}^{q} b_{2,k} t^{2j}-c_{22}|t|^{\alpha_{22}}+o(|t|^{\alpha_{22}}),\\
C_{12}(t)&=C_{21}(t)=\rho \sigma_1\sigma_2\Big(1-\sum_{k=1}^{q} b_{12,k} t^{2j}-c_{12}
|t|^{\alpha_{12}}+o(|t|^{\alpha_{12}})\Big),
\end{split}
\end{align*}
where $q$ is a positive integer and $\alpha_{ii} \in (2q, 2q+2)$. 
Then, the $q$th derivative process $\mathbf{X}^{(q)}:=(X_1^{(q)},X_2^{(q)})^\top$ would satisfy Condition ($\mb A1$) with smoothness parameters $(\alpha_{11}-2q, \alpha_{22}-2q)^\top$. Thus, the framework proposed in our paper can be extended to smooth bivariate Gaussian fields via estimating the fractal indices of their derivative processes.  
\end{remark}

\section{An example: nonsmooth bivariate Mat\'{e}rn processes on $\mathbb{R}$}
\label{sec: an example_bivariate Matern field}
The Mat\'{e}rn correlation function $M(h|\nu,a)$ on $\mathbb R^N$,
where $a>0,\nu>0$ are scale and smoothness parameters, is widely used
to model covariance structures in spatial statistics. It is defined as
\begin{equation}
M(h|\nu,a):=\frac{2^{1-\nu}}{\Gamma(\nu)}(a|h|)^\nu K_\nu(a|h|),
\quad h \in \mathbb R^N,\nonumber
\end{equation}
where $K_\nu$ is a modified Bessel function of the second kind.
Recently, Gneiting, Kleiber and Schlather \cite{Gneiting_Kleiber_Schlather2010} introduced
the full bivariate Mat\'{e}rn field $\mathbf{X}
=\{(X_1(s),X_2(s))^\top,\,  \, s\in \R^N\}$, which is an
$\mathbb{R}^2$-valued Gaussian random field on $\mathbb{R}^N$
with zero mean and matrix-valued covariance function:
\begin{equation}\label{4.2}
\mb C(h)=\left(\begin{array}{ll}
C_{11}(h)&C_{12}(h)\\
C_{21}(h)&C_{22}(h)
\end{array}
\right),
\end{equation}
where $C_{ij}(h):=\E[X_i(s+h)X_j(s)]$ are specified by
\begin{eqnarray}
\begin{split}
C_{11}(h)&=\sigma_1^2M(h|\nu_{11},a_{11}),\\
C_{22}(h)&=\sigma_2^2M(h|\nu_{22},a_{22}),\\
C_{12}(h)&=C_{21}(h)=\rho\sigma_1\sigma_2M(h|\nu_{12},a_{12}).
\end{split}\nonumber
\end{eqnarray}
A necessary and sufficient condition for $C(h)$ in (\ref{4.2})
to be valid is given by \cite{Gneiting_Kleiber_Schlather2010}.
We assume that the parameters $\nu_{ij}, a_{ij}, \sigma_i$, $
(i, j = 1, 2)$ and $\rho$ satisfy the condition in Theorem 3 of \cite{Gneiting_Kleiber_Schlather2010}, as well as our condition
(\ref{Eq:valid}).

To apply the results in Section 3, we focus on the case of $N=1$
and $0<\nu_{11},\nu_{22}<1$. Then, $\mathbf{X}
=\{(X_1(s),X_2(s))^\top,\,  \, s\in \R\}$ is a stationary
bivariate Gaussian process with nonsmooth sample functions.
For simplicity, we call $\mathbf{X}$ a bivariate Mat\'{e}rn process.

Recall that the Mat\'{e}rn correlation function has the following asymptotic expansion at $h=0$,
\begin{align} \label{Eq:Matern}
M(h|\nu,a)=1-b_1|h|^{2\nu}+b_2|h|^2+O(|h|^{2+2\nu}), \ \ \ \ \hbox{ as }\ |h|\to 0,
\end{align}
where $b_1$ and $b_2$ are explicit constants depending only on $\nu$ and $a$
(Eq. (\ref{Eq:Matern}) follows from (9.6.2) and (9.6.10) in
\cite{Abramowitz_Stegun_1972}). Therefore, ({$\mb A3$}) is
satisfied with $\beta_{ij} = 2 - \nu_{ij}$ for $i, j = 1, 2$.
Moreover, one can check that the regularity condition (${\mb A2}$)
regarding the fourth derivatives of the covariance function
is also satisfied (see the proof in Appendix C).

According to (\ref{Eq:Matern}) and the fact that 
$\sum_{j,k=-1}^1a_ja_k|k-j|^2 = 0$, we have
\begin{align}
\label{sigma_n,12^uv(h)-matern}
&\sigma_{n,ii}^{uu}(0):={\rm E}(Y_{n,i}^u(0))^2=n^{2\nu_{ii}}
\sum_{j,k=-1}^1a_ja_k C_{ii}\bigg(\frac{(k-j)u}{n}\bigg)\nonumber\\
&=-b_1\sigma_{ii}^2\sum_{j,k=-1}^1a_ja_k|k-j|^{2\nu_{ii}}u^{2\nu_{ii}}+O(n^{-2}).
\end{align}
Observe that, unlike (\ref{Eq:Ztau}), the constants $\beta_{ij}
= 2 - \nu_{ij}$ do not appear in
\eqref{sigma_n,12^uv(h)-matern} because the related terms sum to 0.
Consequently, we can prove the following results, which are stronger 
than what can be obtained by directly applying  Theorems 
\ref{thm: Bias} $\sim$ \ref{thm: Joint normality of alpha_11_22} 
to bivariate Mat\'{e}rn processes. Their proofs are modifications 
of those of Theorems \ref{thm: Bias} $\sim$
\ref{thm: Joint normality of alpha_11_22} in Section 6 and 
will be omitted.

\begin{proposition}[Bias] For the bivariate Mat\'{e}rn process $\mb X$ with 
$0<\nu_{11},\nu_{22}<1$, the bias of $\hat \nu_{ii}$ is
\label{coro: Bias-matern}
\begin{align*}
{\rm E}\big[\hat \nu_{ii}-\nu_{ii}\big]= O(n^{-1}),\ \  i=1,2.
\end{align*}
\end{proposition}

For the next proposition, we write $\bb\nu=(\nu_{11},\nu_{22})^\top$ 
and $\hat {\bb\nu}
=(\hat \nu_{11},\hat \nu_{22})^\top$.
\begin{proposition}[Mean square error matrix]
\label{coro: Mean Square Error Matrix-matern}
For the bivariate Mat\'{e}rn process $\mb X$ with 
$0<\nu_{11},\nu_{22}<1$, 
if $(\nu_{11} +\nu_{22})/2=\nu_{12}$, then
\begin{align*}
&{\rm E}[(\hat {\bb\nu}-\bb\nu)(\hat {\bb\nu}-\bb\nu)]^\top
=\left(
\begin{array}{ll}
O(n^{-1})&O(n^{-1})\\
O(n^{-1})&O(n^{-1})
\end{array}
\right);
\end{align*}
if $(\nu_{11}+\nu_{22})/2<\nu_{12}$, we have
\begin{align*}
&{\rm E}[(\hat{\bb\nu}-\bb\nu)(\hat {\bb\nu}-\bb\nu)]^\top
=\left(
\begin{array}{ll}
O(n^{-1})&o(n^{-1})\\
o(n^{-1})&O(n^{-1})
\end{array}
\right).
\end{align*}
\end{proposition}

\begin{proposition}[Asymptotic distribution]
\label{coro: Joint normality of alpha_11_22-matern}
For the bivariate Mat\'{e}rn process $\mb X$ with 
$0<\nu_{11},\nu_{22}<1$,
\begin{align*}
\sqrt{n}\left(
\begin{array}{l}
\hat\nu_{11}-\nu_{11}\\
\hat \nu_{22}-\nu_{22}
\end{array}
\right)\xrightarrow d \mathcal{N}\left(\left(\begin{array}{l}
0\\
0
\end{array}
\right),
\left(
\begin{array}{ll}
\widetilde {\mathbf L}_1^\top \mb\Phi_{0,11}\widetilde 
{\mathbf L}_1&\widetilde {\mb L}_1^\top \mb\Phi_{0,12}
\widetilde {\mathbf L}_2\\
\widetilde {\mathbf L}_2^\top \mb\Phi_{0,21} \widetilde{\mathbf L}_1
&\widetilde {\mathbf L}_2^\top \mb\Phi_{0,22}\widetilde {\mathbf L}_2
\end{array}
\right)
\right).
\end{align*}
Specifically, if  $(\nu_{11}+\nu_{22})/{2}<\nu_{12}$,
$\sqrt{n} \hat \nu_{11}$ and
$\sqrt{n}\hat \nu_{22}$
are asymptotically independent.
\end{proposition}

\section{Simulation Study}
\label{sec: simulation}
In this section, we simulate data from a nonsmooth bivariate Mat\'ern 
process and illustrate that when $(\nu_{11}+\nu_{22})/{2} 
= \nu_{12}$, the decay rates of the bias and mean square error matrix for 
$\hat \nu_{11}$ and $\hat \nu_{22}$ are $n^{-1}$. Then, we compare 
with the case when $(\nu_{11}+\nu_{22})/{2} < \nu_{12}$.

We take $\nu_{11}=0.2$, $\nu_{22}=0.7$, $\nu_{12}=0.45$, $\rho =0.5$, 
$\sigma_1^2=\sigma_2^2=1$ and $a_{11}=a_{22}=a_{12}=1$. We simulated 
the corresponding bivariate Mat\`ern process on regular grids within the 
interval $[0,1]$, where the length of the grid  was set to $1/n$ with 
$n=200,210,220,\ldots,1000$. For each $n$, we used generalized least squares (\textit{abbr. GLS}) to obtain the estimators 
of the fractal indices, say $(\hat\nu_{11}$, $\hat\nu_{22})$ (see, e.g., \cite{Kent_Wood_1997}). Here,  we fixed the number of dilations to $m=50$. The  weight matrix $\Omega_i = (\omega_i^{uv})_{u,v=1}^m$ of  the GLS estimator with a Mat\'ern covariance function is given by
\begin{align*}
&\omega^{uv}_i=\frac{2}{(n-2u+1)(n-2v+1)}\frac{\sum_{h=u}^{n-u}\sum_{\ell=v}^{n-v}(\sum_{j,k=-1}^1 a_ja_k|h-\ell+kv-ju|^{2\nu_{ii}})^2}{(\sum_{j,k=-1}^1 a_ja_k|k-j|^{2\nu_{ii}})^2u^{2\nu_{ii}}v^{2\nu_{ii}}},
\end{align*} 
which can be approximated by plugging in the ordinary least squares estimators of $\nu_{ii}, i =1,2$. To evaluate the efficiency of the estimators, we repeated the above procedure $1000$ times independently.

The $95\%$ confidence intervals for $(\nu_{11},\nu_{22})^\top$ with varying $n$ are 
shown in FIG \ref{Fig1} (a). FIG \ref{Fig1} (c) and (e) show how the 
bias,  marginal variances and cross covariance decrease when $n$ increases 
from $200$ to $1000$.  By fitting the natural logarithm of the absolute value of the bias, marginal 
variances and absolute values of the cross covariance with respect to $\ln n$, we find the power of the decay rate for each is very 
close to $-1$. This is consistent with the conclusions in Proposition
\ref{coro: Bias-matern} and Proposition \ref{coro: Mean Square Error Matrix-matern} when $(\nu_{11}+\nu_{22})/2 
= \nu_{12}$.

Further, we show how the decay rate changes if $(\nu_{11}+\nu_{22})/2 
< \nu_{12}$.  Fixing all previously assigned parameters but 
setting $\nu_{12}$ to $0.6$, we rerun the simulation and repeat the 
estimation procedures. The results are shown on the right side of 
FIG \ref{Fig1}, where we can see that the results are mostly the same,
but the cross covariance decays much faster than $n^{-1}$. Indeed, the power of the decay rate is approximately $-1.5$, which is consistent 
with the conclusion in Proposition \ref{coro: Mean Square Error Matrix-matern}.

\begin{figure}[!htbp]
\centering
\subfigure[]{\includegraphics[width=0.48\textwidth]{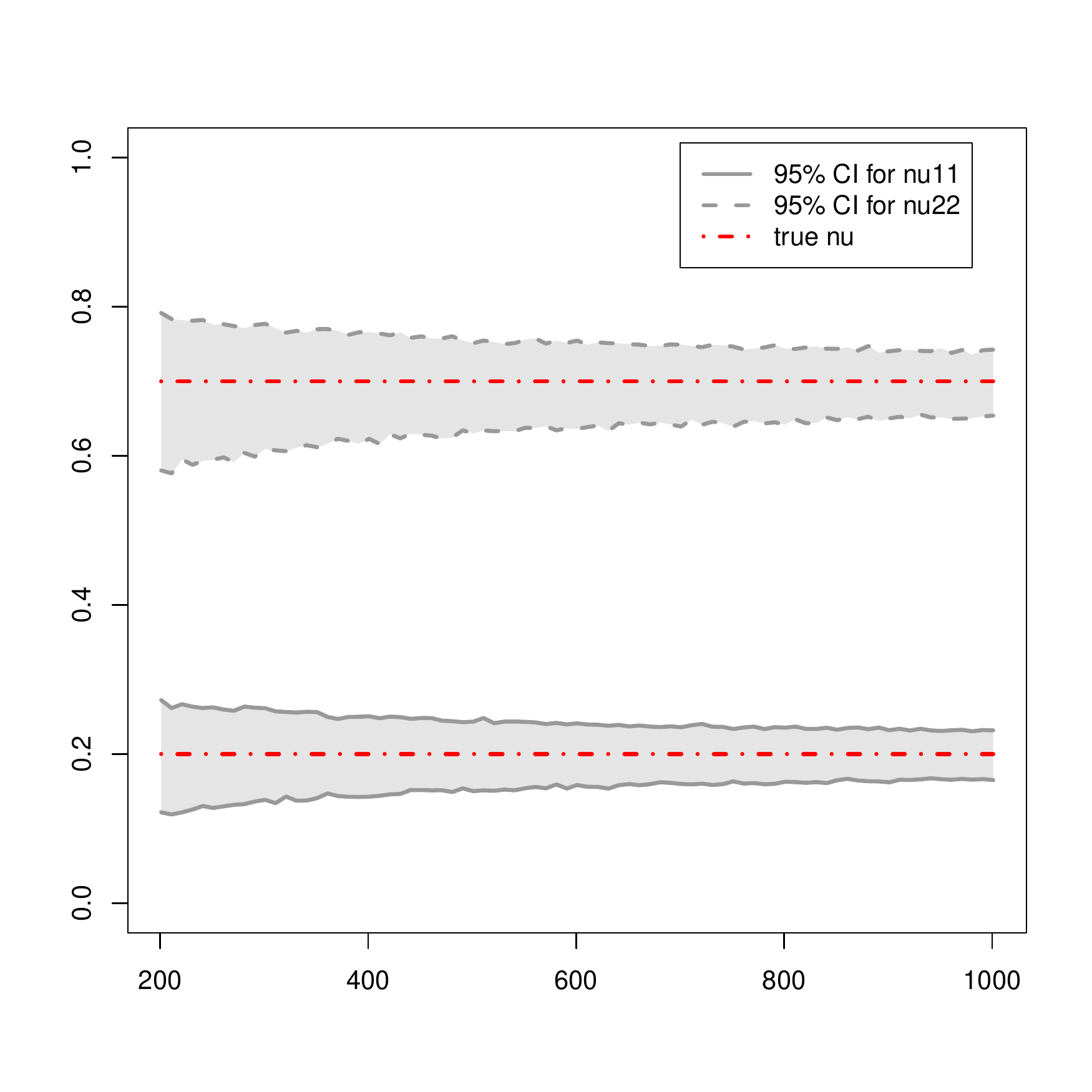}}
\subfigure[]{\includegraphics[width=0.48\textwidth]{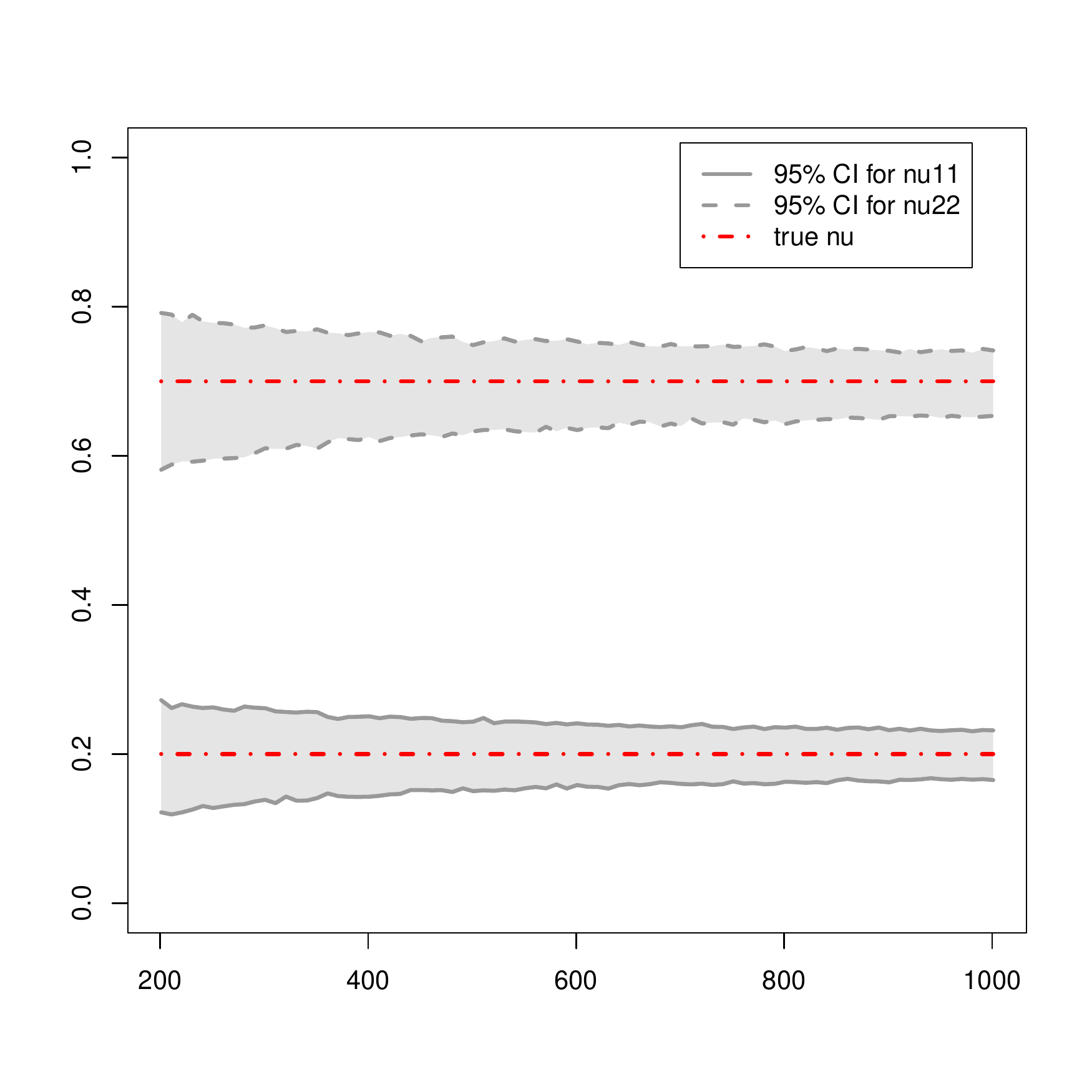}}\\
\subfigure[]{\includegraphics[width=0.48\textwidth]{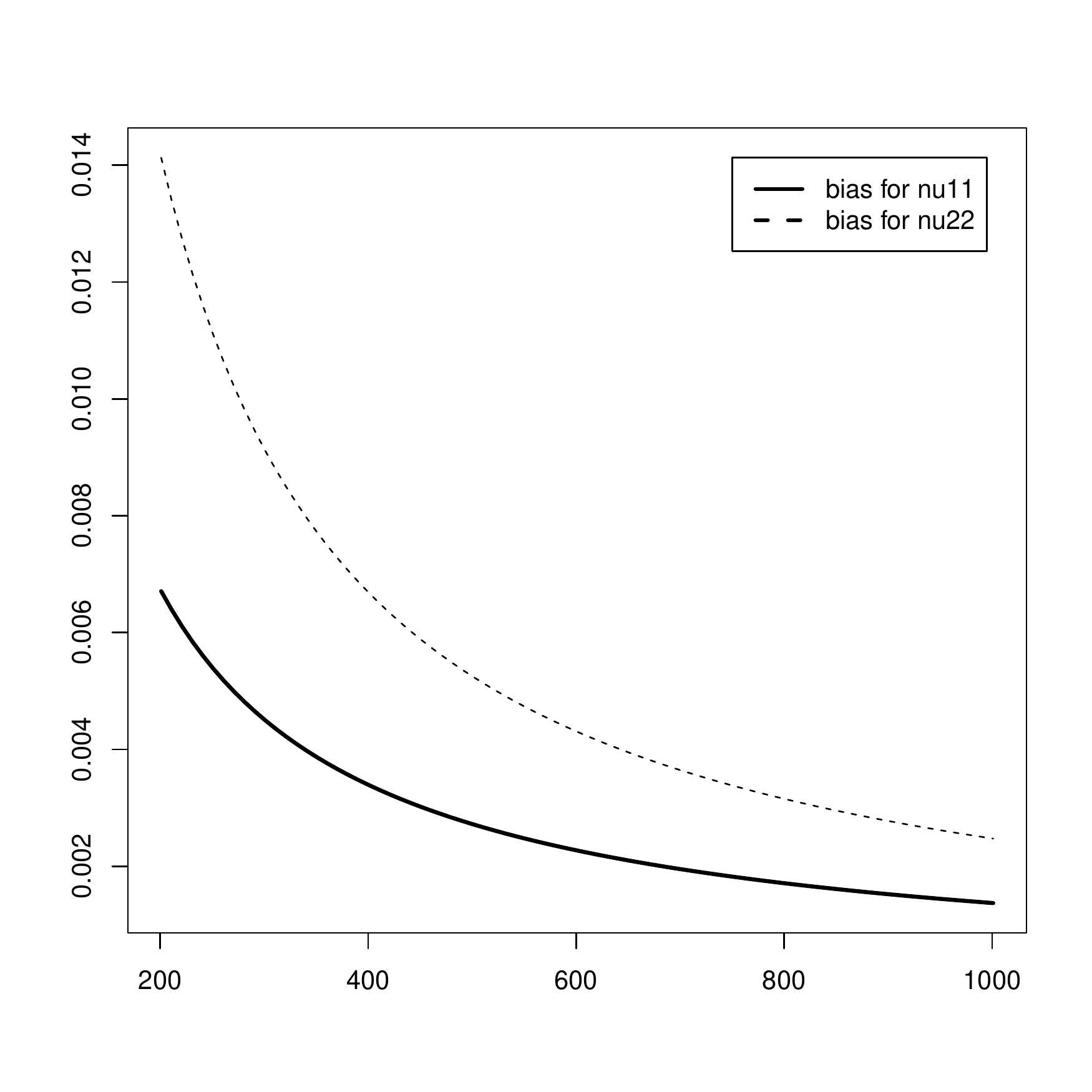}}
\subfigure[]{\includegraphics[width=0.48\textwidth]{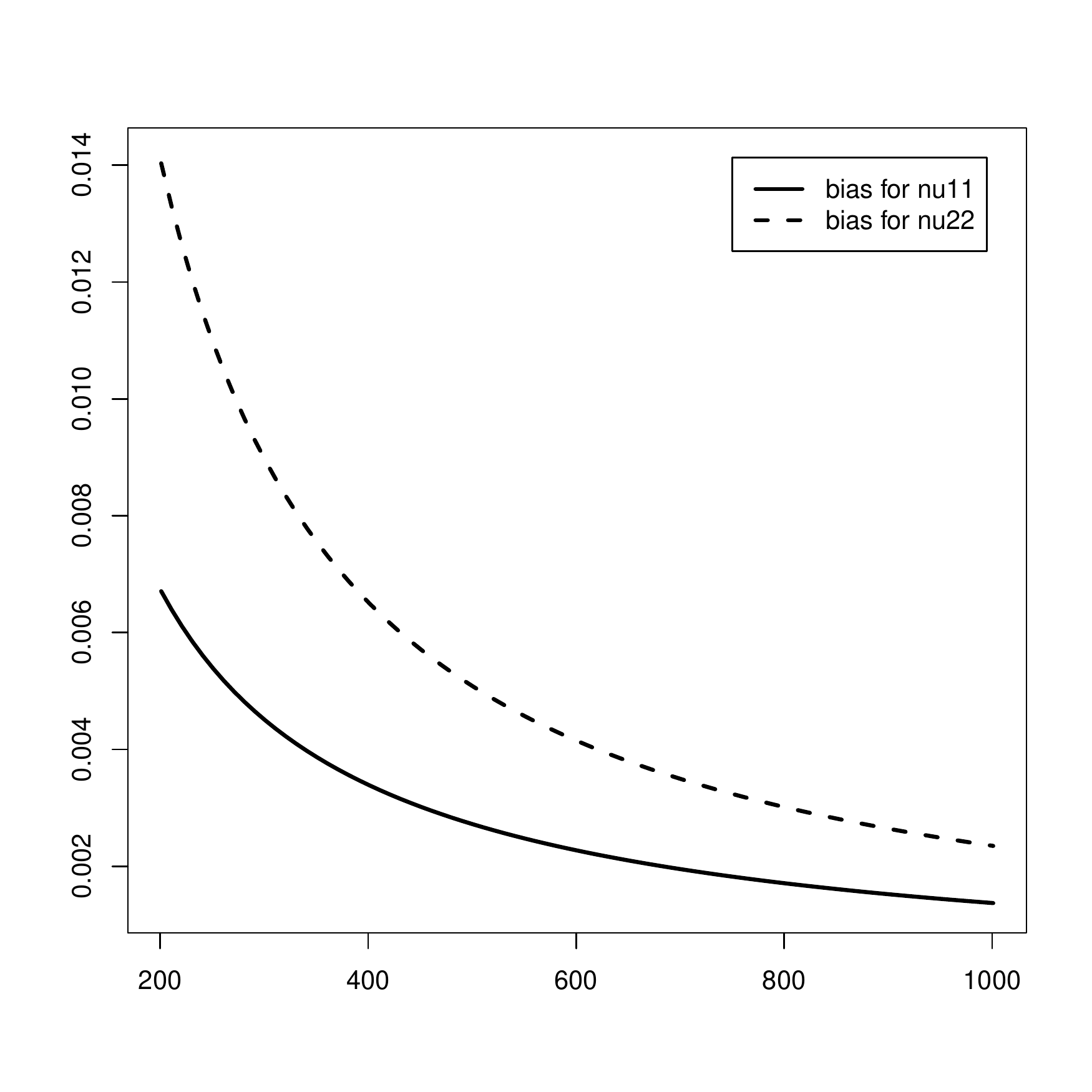}}\\
\subfigure[]{
\includegraphics[width=0.48\textwidth]{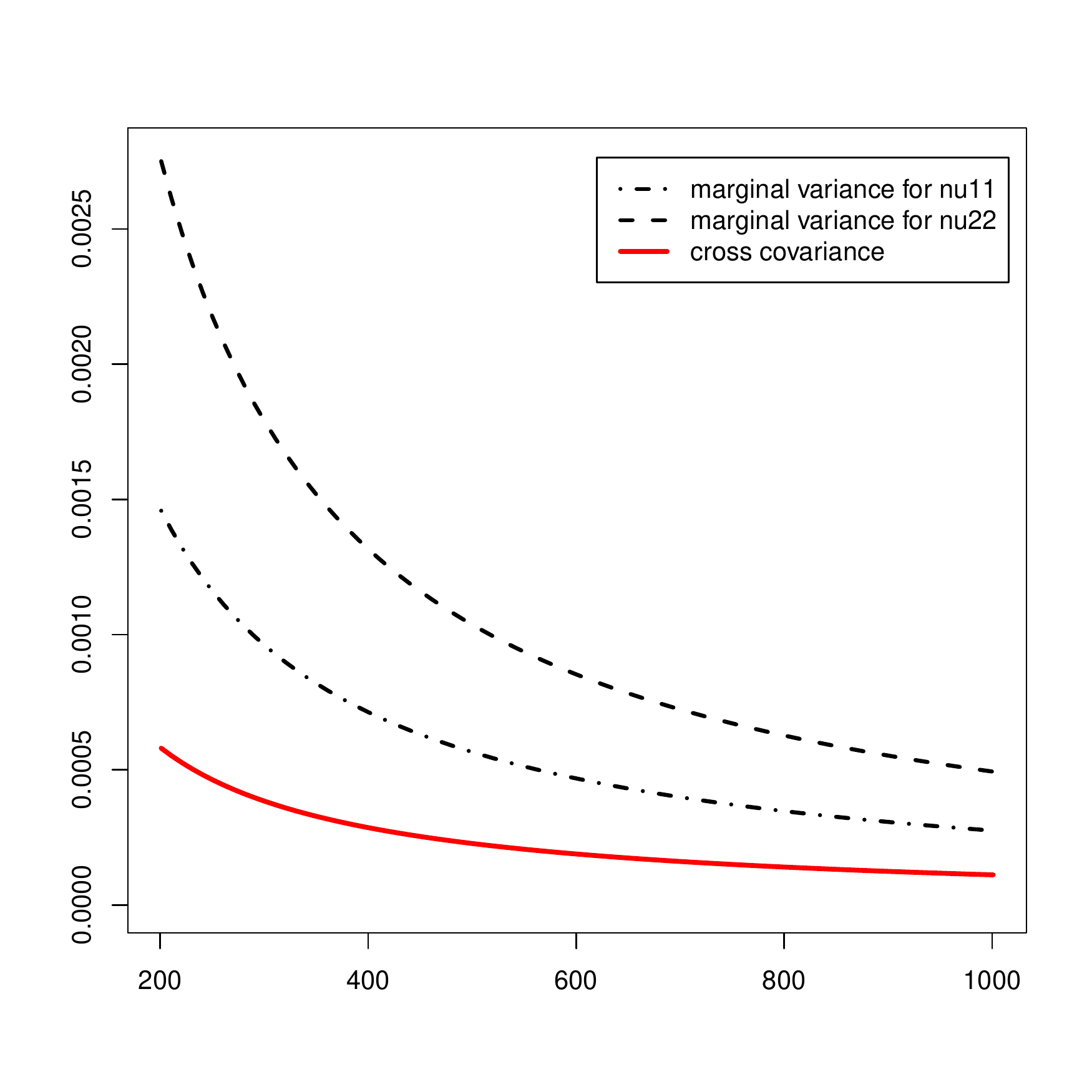}}
\subfigure[]{
\includegraphics[width=0.48\textwidth]{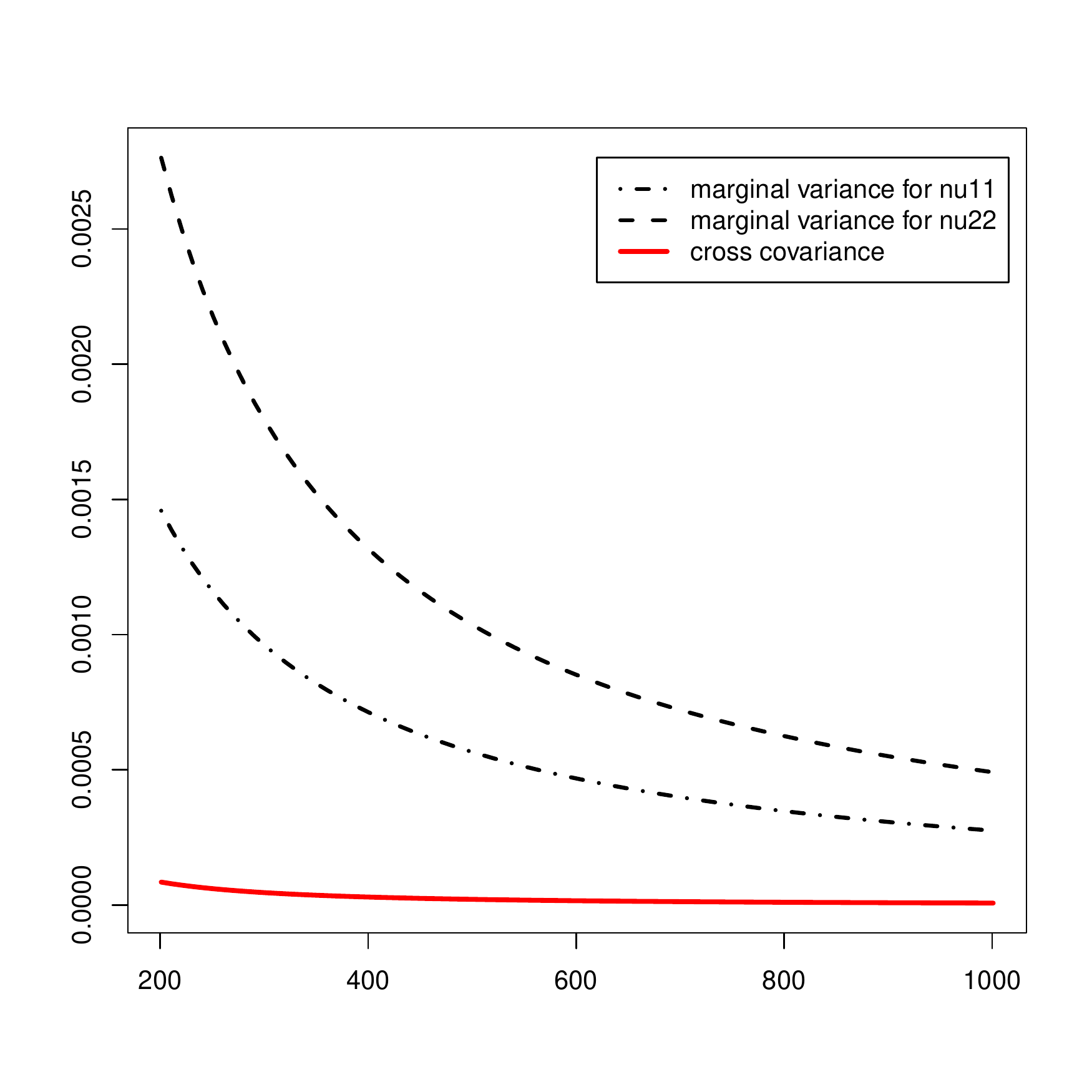}}\\
\caption{Confidence intervals, absolute value of the bias, marginal 
variances and absolute value of the cross covariance for $(\hat \nu_{11},
\hat \nu_{22})$ with varying $n$. The plots on the left side
(i.e., a, c, e) show the results  for $(\nu_{11}
+\nu_{22})/2=\nu_{12}$, whereas those on the right side (i.e., 
b, d, f) correspond to the situation where $(\nu_{11}+\nu_{22})/{2}
<\nu_{12}$.} \label{Fig1}
\end{figure}

\section{Proof of the main results}
\label{sec_proof of main results}
To prove Theorems \ref{thm: Bias} $\sim$
\ref{thm: Joint normality of alpha_11_22}, we make
use of the following key lemma.
\begin{lemma}
\label{lem_1}
For $u =1,\ldots,m$ and $ i=1,2$, let $T_{n,i}^u=(\bar Z_{n,i}^u
-{\rm E}\bar Z_{n,i}^u)/{\rm E}\bar Z_{n,i}^u$.
Then, for any $k\in \mathbb{Z}^+$, there exist positive and finite constants
$C_3$ and $C_4$ (which may depend on $u$ and $k$) such that for all $n \ge 1$ and 
 $\xi>0$,
\begin{align*}
{\rm E} \big[\big|\ln(1+T_{n,i}^u)\big|^k;\, |T_{n,i}^u|> \xi]
\leq C_3\, e^{-C_4\xi\sqrt{n}}.
\end{align*}
\end{lemma}

The proof of Lemma \ref{lem_1} is given at the end of this section.
Now, we proceed to prove our main theorems.
\begin{proof}[Proof of Theorem \ref{thm: Bias}]
Recall that $T_{n,i}^u=(\bar Z_{n,i}^u-{\rm E}\bar Z_{n,i}^u)/{\rm E}\bar Z_{n,i}^u.$ 
Then,
\begin{align}
\label{alpha_ii div two}
\hat \alpha_{ii} = \sum_{u=1}^m L_{u,i}\ln(1+T_{n,i}^u)+
\sum_{u=1}^m L_{u,i} \ln \E{\bar Z_{n,i}^u}.
\end{align}
It follows from (\ref{Eq:Ztau}) that
\begin{align*}
{\rm E}\bar Z^u_{n,i}=\sigma_{n,ii}^{uu}(0)=C_i\, u^{\alpha_{ii}}
(1+O(n^{-\beta_{ii}})),
\end{align*}
Hence, using the conditions
on $L_{u,i}$ in \eqref{L_u,i assumption}, we conclude
\begin{align}
\label{sum_L_E_Z}
\sum_{u=1}^mL_{u,i}\ln{\rm E}\bar Z^u_{n,i}=\alpha_{ii}+O(n^{-\beta_{ii}}).
\end{align}

Next, we estimate the first sum in \eqref{alpha_ii div two}.  By Taylor's expansion, 
we obtain
\begin{align*}
\ln(1+T_{n,i}^u)=T_{n,i}^u-\frac{1}{2}(T_{n,i}^u)^2+R_{n,i}^u,
\end{align*}
where $R_{n,i}^u$ is the residual term.
Hence,
\begin{align*}
{\rm E}[\ln(1+T_{n,i}^u)]={\rm E}\Big[T_{n,i}^u-\frac{1}{2}(T_{n,i}^u)^2+R_{n,i}^u\Big]
=-\frac{1}{2}{\rm E}(T_{n,i}^u)^2+\E R_{n,i}^u.
\end{align*}
By Theorem \ref{thm: asymptotic variance of Z_1&Z_2}, it is easy to verify that
\begin{align}
\label{ZTsquare_order}
{\rm E}(T_{n,i}^u)^2 = O(n^{-1}).
\end{align}
Using the fact that if $|T_{n,i}^u|\leq \xi\leq 1/2$, then $|R_{n,i}^u|\leq \xi (T_{n,i}^u)^2$, we have
\begin{align}
\label{ZR_T}
{\rm E}[|R_{n,i}^u|;\, |T_{n,i}^u|\leq \xi]\leq \xi {\rm E}[(T_{n,i}^u)^2;\,|T_{n,i}^u|\leq \xi]=O(n^{-1}),
\end{align}
where the last equality follows from \eqref{ZTsquare_order}.
On the other hand, by applying Lemma \ref{lem_1}, 
the Cauchy-Schwarz inequality and \eqref{ZTsquare_order}, we obtain
\begin{align}
\label{ZR_T2}
&{\rm E} [|R_{n,i}^u|;\, |T_{n,i}^u|> \xi]\nonumber\\
&\leq {\rm E} [|\ln(1+T_{n,i}^u)|+|T_{n,i}^u|+\frac{1}{2}(T_{n,i}^u)^2; \, |T_{n,i}^u|> \xi]=O(n^{-1}).
\end{align}
By combining \eqref{ZTsquare_order}, \eqref{ZR_T} and \eqref{ZR_T2}, we obtain
\begin{align}
\label{sum_L_E_log(1+T)}
{\rm E} [\ln(1+T_{n,i}^u)]=O(n^{-1}).
\end{align}
This, together with \eqref{alpha_ii div two} and \eqref{sum_L_E_Z}, proves Theorem \ref{thm: Bias}.
\end{proof}

\begin{proof}[Proof of Theorem \ref{thm: Mean Suare Error Matrix}] 
For $i = 1, 2$, we expand ${\rm E}\big[(\hat \alpha_{ii}
-\alpha_{ii})^2\big]$ as follows.
\begin{align}
\label{mean square error of alpha_ii}
{\rm E}\big[(\hat \alpha_{ii}-\alpha_{ii})^2\big]
&=\sum_{u=1}^m\sum_{v=1}^mL_{u,i}L_{v,i}{\rm E}\Big[\big(\ln (1+T_{n,i}^u)
+\ln{\rm E}\bar Z_{n,i}^u-\alpha_{ii}\ln u \big)\nonumber\\
&\quad \qquad \times \big(\ln (1+T_{n,i}^v)
+\ln{\rm E}\bar Z_{n,i}^v-\alpha_{ii}\ln v\big)\Big]\nonumber\\
&=\sum_{u=1}^m\sum_{v=1}^mL_{u,i}L_{v,i}{\rm E} 
\big[\ln (1+T_{n,i}^u) \ln (1+T_{n,i}^v)\big]\nonumber\\
&\quad+\sum_{u=1}^m\sum_{v=1}^mL_{u,i}L_{v,i}
{\rm E} \big[\ln (1+T_{n,i}^u)\big] 
\big(\ln{\rm E}\bar Z_{n,i}^v-\alpha_{ii}\ln v\big)\nonumber\\
&\quad+\sum_{u=1}^m\sum_{v=1}^mL_{u,i}L_{v,i} 
\big(\ln{\rm E}\bar Z_{n,i}^u-\alpha_{ii}\ln u \big)
{\rm E}\big[\ln (1+T_{n,i}^v)\big] \nonumber\\
&\quad+\sum_{u=1}^m\sum_{v=1}^mL_{u,i}L_{v,i} 
\big(\ln{\rm E}\bar Z_{n,i}^u-\alpha_{ii}\ln u \big)
\big(\ln{\rm E}\bar Z_{n,i}^v-\alpha_{ii}\ln v\big)\nonumber\\
&\triangleq \text{I}+\text{II}+\text{III}+\text{IV}.
\end{align}
By  \eqref{sum_L_E_Z} and \eqref{sum_L_E_log(1+T)}, we have
\begin{align}
\label{sum_log(1+T)_E(Z)-alpha}
&\text{II}=O(n^{-1-\beta_{ii}}),\ \ \text{III}=O(n^{-1-\beta_{ii}}),\ \ \text{IV}=O(n^{-2\beta_{ii}}).
\end{align}
To bound the first term I in \eqref{mean square error of alpha_ii}, 
we take $\xi= 1/2$ and decompose the probability space into the union 
of the following four disjoint events,  
$\{|T_{n,i}^u|\leq\xi,\, |T_{n,i}^v|\leq \xi\}$, 
$\{|T_{n,i}^u|> \xi,\, |T_{n,i}^v|\leq \xi\}$, 
$\{|T_{n,i}^u|\leq \xi,\, |T_{n,i}^v|> \xi\}$, and 
$\{|T_{n,i}^u|> \xi,\, |T_{n,i}^v|>\xi\}$.
\begin{itemize}
\item[i).] For the event $\{|T_{n,i}^u|\leq\xi,\, |T_{n,i}^v|\leq\xi\}$, 
we use the elementary inequality $|\ln(1+x)|  \le 2 |x|$ for all 
$|x|\le \xi$  to derive
\begin{align*}
&|\ln (1+T_{n,i}^u)\ln (1+T_{n,i}^v)|
\leq 4\,|T_{n,i}^u||T_{n,i}^v|.
\end{align*}
It follows from the Cauchy-Schwarz inequality and Theorem 
\ref{thm: asymptotic variance of Z_1&Z_2} that
\begin{align*}
{\rm E}\big[\ln (1+T_{n,i}^u)\ln (1+T_{n,i}^v);\, |T_{n,i}^u|\leq\xi,|T_{n,i}^v|\leq\xi\big]=O(n^{-1}).
\end{align*}
\item[ii).] By Lemma \ref{lem_1}, we have
\begin{align*}
&{\rm E}\big[\big|\ln (1+T_{n,i}^u)\ln (1+T_{n,i}^v)\big|;\, |T_{n,i}^u|>\xi,|T_{n,i}^v|\leq\xi \big]\\
&\leq \ln 2\, {\rm E}[|\ln (1+T_{n,i}^u)|;\, |T_{n,i}^u|>\xi]\\
&=o(n^{-1}).
\end{align*}

\item[iii).] As in ii), we have
\begin{align*}
&{\rm E}[\ln (1+T_{n,i}^u)\ln (1+T_{n,i}^v);\, |T_{n,i}^u|\leq\xi,|T_{n,i}^v|>\xi]\nonumber\\
&\leq \ln 2\, {\rm E}[|\ln (1+T_{n,i}^v)|;\, |T_{n,i}^v|>\xi]\\
&=o(n^{-1}).
\end{align*}
\item[iv).] By Lemma \ref{lem_1} and the Cauchy-Schwarz inequality, we have
\begin{align*}
&{\rm E}\big[\big|\ln (1+T_{n,i}^u)\ln (1+T_{n,i}^v)\big|;\,
|T_{n,i}^u|>\xi,\, |T_{n,i}^v|>\xi\big]\\
&\leq \sqrt{{\rm E}\big[\ln^2 (1+T_{n,i}^u);\, |T_{n,i}^u|>\xi]}\sqrt{{\rm E}\big[\ln^2(1+T_{n,i}^v);\, |T_{n,i}^v|>\xi\big]}\\
&=o(n^{-1}).
\end{align*}
\end{itemize}
By combining i)-(iv) above, we see that
\begin{align}
\label{log(1+T)log(1+T)}
\text{I}
=O(n^{-1}).
\end{align}
By \eqref{sum_log(1+T)_E(Z)-alpha} and \eqref{log(1+T)log(1+T)}, we have
\begin{align*}
{\rm E}\big[(\hat \alpha_{ii}-\alpha_{ii})^2 \big]
=O(n^{-1})+ O(n^{-2\beta_{ii}}).
\end{align*}

Next, we study the cross term ${\rm E}\big[(\hat \alpha_{11}-\alpha_{11})(\hat \alpha_{22}-\alpha_{22})\big]$, which can be written as
\begin{align}
\label{cross error of alpha_11, alpha_22}
&{\rm E}\big[(\hat \alpha_{11}-\alpha_{11})(\hat \alpha_{22}-\alpha_{22})\big]
\nonumber\\
&=\sum_{u=1}^m\sum_{v=1}^mL_{u,1}L_{v,2}{\rm E}\big[\ln (1+T_{n,1}^u)\ln (1+T_{n,2}^v)\big]\nonumber\\
&\quad+\sum_{u=1}^m\sum_{v=1}^mL_{u,1}L_{v,2}{\rm E}\big[\ln (1+T_{n,1}^u)\big]
(\ln{\rm E}\bar Z_{n,2}^v-\alpha_{22}\ln v)\nonumber\\
&\quad+\sum_{u=1}^m\sum_{v=1}^mL_{u,1}L_{v,2}
\big(\ln{\rm E}\bar Z_{n,1}^u-\alpha_{11}\ln u\big)\,
{\rm E}\big[\ln (1+T_{n,2}^v)\big]\nonumber\\
&\quad+\sum_{u=1}^m\sum_{v=1}^mL_{u,1}L_{v,2}
\big(\ln{\rm E}\bar Z_{n,1}^u-\alpha_{11}\ln u\big)
\big(\ln{\rm E}\bar Z_{n,2}^v-\alpha_{22}\ln v\big)\nonumber\\
&\triangleq \text{I}+\text{II}+\text{III}+\text{IV}.
\end{align}
Applying similar arguments as used in evaluating 
$\E\big[(\hat \alpha_{ii}-\alpha_{ii})^2\big]$,
we obtain
\begin{align} \label{Eq:III}
&\text{II}=O(n^{-1-\beta_{22}}),\ \ \ \text{III}=O(n^{-1-\beta_{11}}),\ \ \ \text{IV}=O(n^{-\beta_{11}-\beta_{22}}).
\end{align}
In order to bound the term \text{I}, 
we distinguish two cases.
\begin{itemize}
\item[1).] If $(\alpha_{11}+\alpha_{22})/2=\alpha_{12}$, then by a similar 
argument as that for proving \eqref{log(1+T)log(1+T)} (using Lemma \ref{lem_1} 
and Theorem \ref{thm: asymptotic variance of Z_1&Z_2})
and the fact that $\mb \Phi_{0,12}\neq \mb 0$, we have
\begin{align*}
\sum_{u=1}^m\sum_{v=1}^mL_{u,1}L_{v,2}{\rm E}\big[\ln (1+T_{n,1}^u)\ln (1+T_{n,2}^v)\big]=O(n^{-1}).
\end{align*}
This, together with \eqref{cross error of alpha_11, alpha_22} and \eqref{Eq:III}, 
implies
\begin{align*}
&{\rm E}\big[(\hat \alpha_{11}-\alpha_{11})
(\hat \alpha_{22}-\alpha_{22})\big]\nonumber\\
&=O(n^{-1})+O(n^{-1-\beta_{11}})+O(n^{-1-\beta_{22}})
+O(n^{-\beta_{11}-\beta_{22}}).
\end{align*}
\item[2).]  If $(\alpha_{11}+\alpha_{22})/2<\alpha_{12}$, then 
by an argument similar to that for proving \eqref{log(1+T)log(1+T)} 
and the fact that $\mb \Phi_{0,12}= \mb 0$, we obtain
\begin{align*}
\sum_{u=1}^m\sum_{v=1}^mL_{u,1}L_{v,2}{\rm E}\big[\ln (1+T_{n,1}^u)
\ln (1+T_{n,2}^v)\big]=o(n^{-1}).
\end{align*}
Consequently, 
\begin{align*}
&{\rm E} \big[(\hat \alpha_{11}-\alpha_{11})(\hat \alpha_{22}-\alpha_{22})\big]
\nonumber\\
&=o(n^{-1})+O(n^{-1-\beta_{11}})+O(n^{-1-\beta_{22}})+O(n^{-\beta_{11}-\beta_{22}}).
\end{align*}
\end{itemize}
Therefore, we have proved  (\ref{Eq:16}) and (\ref{Eq:17}).
\end{proof}

\begin{proof}[Proof of Theorem \ref{thm: Joint normality of alpha_11_22}]
Recall (\ref{Eq:Ztau}) for ${\rm E}\bar Z^u_{n,i}$ and denote
$$\bb \tau= (\tau_{1,1},\ldots,\tau_{m,1},\tau_{1,2},\ldots,\tau_{m,2})^\top.$$
Since 
$\beta_{ii}>1/2$ for $ i =1,2$, we have
$\sqrt{n}({\rm E} \mb {\bar Z}_n-\bb \tau)\rightarrow 0$  as 
$n\rightarrow \infty$. By Theorem
\ref{thm: asymptotic normality of bar Z} and Slutsky's theorem,
we conclude that
\begin{align*}
n^{1/2}(\mb{\bar Z}_n-\bb \tau)\xrightarrow{d} \mathcal{N}_{2m}(\mb 0,\bb\Phi_0), \
\text{ as }\ n\rightarrow \infty.
\end{align*}
Define a mapping $\mathbf f: \mathbb{R}^{2m} \rightarrow \mathbb{R}^2$, 
$\forall \mathbf{x}
\triangleq(x_{1,1},\ldots,x_{m,1},x_{1,2},\ldots,x_{m,2})\in \mathbb{R}^{2m}$,
\begin{align*}
\mathbf f(\mathbf{x}) := \Big(\sum_{u=1}^m L_{u,1}\ln x_{u,1},\sum_{u=1}^m L_{u,2}\ln x_{u,2}\Big)^\top.
\end{align*}
Hence, it is easy to verify that $\mb f(\cdot)$ is continuously 
differentiable, $\hat {\bb\alpha} =\mathbf f(\mb{\bar Z}_n)$ and 
$\bb \alpha =\mb f(\bb \tau)$. By applying the multivariate delta
method \cite[Theorem 8.22]{Lehmann_2006}, we conclude that
\begin{align*}
\sqrt{n} (\hat {\bb\alpha}-\bb\alpha) \xrightarrow d \mathcal{N}(\mb 0, 
\triangledown \mb f(\bb \tau)^\top
\bb \Phi_0 \triangledown \mb f(\bb \tau)),
\end{align*}
where $\triangledown \mb f(\bb \tau)=(\widetilde{\mathbf L}_1^\top, 
\widetilde{\mathbf L}_2^\top)^\top$.
\end{proof}

\begin{proof}[Proof of Lemma \ref{lem_1}]
By H\"older's inequality, we have
\begin{align}
\label{holder ineq}
\begin{split}
&{\rm E} \big[\big|\ln(1+T_{n,i}^u)\big|^k;\, |T_{n,i}^u|> \xi \big]\\
&\leq \sqrt{{\rm E}\big[\ln^{2k}(1+T_{n,i}^u)\big]}
\sqrt{\P( |T_{n,i}^u|> \xi)}.
\end{split}
\end{align}
First, we establish an upper bound for $\P( |T_{n,i}^u|> \xi)$.
By (\ref{sigma_n,ii^uv(h)}), we see that the covariance matrix of the 
random vector 
$\mb Y_{n,i} =(Y_{n,i}^u(1),\ldots,Y_{n,i}^u(n))^\top$ is 
$ \bb\Sigma_{n,i}
=(\sigma_{n,ii}^{uu}(j-k))_{j,k=1}^n$. Let $\bb\Lambda_{n,i}
={\rm diag}(\lambda_{j,i})_{j=1}^n$ be the diagonal matrix whose
diagonal entries are the eigenvalues of
$\bb\Sigma_{n,i}$, and let $\mb U=(U_1,\ldots,U_n)^\top$, where $U_j\stackrel{iid}
{\sim} \mathcal{N}(0,1),\ j=1,\ldots,n$. Then, we have
\begin{align*}
\bar Z_{n,i}^u=\frac{1}{n}\mb Y_{n,i}^\top \mb Y_{n,i}
\stackrel{d}{=}\frac{1}{n}\mb U^\top\bb\Lambda_{n,i}\mb U.
\end{align*}
Since $n{\rm E}\bar Z_{n,i}^u={\rm E}(\mb U^\top
\bb\Lambda_{n,i}\mb U)={\rm trace}(\bb\Lambda_{n,i})$, we
apply the Hanson and Wright inequality \citep{Hanson_1971}
to the tail probability of the quadratic forms to obtain
\begin{align}
\label{HansonWrightBound}
&\P\big( |T_{n,i}^u|> \xi\big)=
\P\big( |\mb U^\top
\bb\Lambda_{n,i}\mb U-{\rm trace}(\bb\Lambda_{n,i})|
> {\rm trace}(\bb\Lambda_{n,i}\big) \xi)\nonumber\\
\leq & \exp\bigg\{-{\rm {min}}\bigg(C_5\xi\frac{{\rm trace}(\bb\Lambda_{n,i})}
{\|\bb\Lambda_{n,i}\|_2},
C_6 \xi^2\frac{({\rm trace}(\bb\Lambda_{n,i}))^2}
{\|\bb\Lambda_{n,i}\|_F^2}\bigg)\bigg\},
\end{align}
where $\|\bb\Lambda_{n,i}\|_2$ and $\|\bb\Lambda_{n,i}\|_F$
are the $\ell_2$ norm and Frobenius norm of $\bb\Lambda_{n,i}$,
respectively, and $C_5,\, C_6$ are positive constants independent
of $\bb\Lambda_{n,i}$, $n$ and $\xi$.

Note that
\begin{align*}
\|\bb\Lambda_{n,i}\|_F^2={\rm trace}(\bb\Lambda_{n,i}^2)
={\rm trace}(\bb\Sigma_n^2)=\sum_{j=1,k=1}^n(\sigma_{n,ii}^{uu}(j-k))^2,
\end{align*}
and
\begin{align*}
\phi_{n,ii}^{uu}=\text{Var}(\bar Z_{n,ii}^u)=\frac{2}{n^2}\sum_{j=1,k=1}^n(\sigma_{n,ii}^{uu}(j-k))^2.
\end{align*}
By Theorem \ref{thm: asymptotic variance of Z_1&Z_2}, we have
\begin{align*}
\|\bb\Lambda_{n,i}\|_F^2=\frac{n^2}{2}\phi_{n,ii}^{uu}
\asymp n\,\phi_{0,ii}^{uu}.
\end{align*}
By combining the above with the facts that $\|\bb\Lambda_{n,i}\|_2\leq
\|\bb\Lambda_{n,i}\|_F$ and ${\rm trace}(\bb\Lambda_{n,i})/n
= {\rm E}\bar Z_{n,i}^u \rightarrow C_i u^{\alpha_{ii}}$ as 
$n \to \infty$, we have
\begin{align*}
\frac{{\rm trace}(\bb\Lambda_{n,i})}{\|\bb\Lambda_{n,i}\|_2}
=\sqrt{n} \frac{{\rm trace}(\bb\Lambda_{n,i})/n}
{\|\bb\Lambda_{n,i}\|_2/\sqrt{n}}\gtrsim 
 u^{\alpha_{ii}}(\phi_{0,ii}^{uu})^{-1/2}\sqrt{n},
\end{align*}
and
\begin{align*}
\frac{({\rm trace}(\bb\Lambda_{n,i}))^2}{\|\bb\Lambda_{n,i}\|_F^2}
\asymp u^{2\alpha_{ii}}(\phi_{0,ii}^{uu})^{-1}n, \
\text{as}\ n\rightarrow \infty.
\end{align*}
Hence, when $n\rightarrow \infty$, \eqref{HansonWrightBound}
decays exponentially with rate $\sqrt{n}$. Consequently,
when $n$ is sufficiently large,
\begin{align}
\label{Tn tail upper bound}
&\P( |T_{n,i}^u|> \xi)\leq e^{-C_0 u^{\alpha_{ii}}
(\phi_{0,ii}^{uu})^{-1/2}\xi\sqrt{n}}.
\end{align}
Next, we prove ${\rm E}\big[\ln^{2k}(1+T_{n,i}^u)\big]$ is
bounded by $C_0n$. It is easy to see that
\begin{align*}
{\rm E}\big[\ln^{2k}(1+T_{n,i}^u)\big] \leq 2^{2k-1}
\Big({\rm E}\ln^{2k}\bar Z_{n,i}^u +\ln^{2k}
({\rm E}\bar Z_{n,i}^u)\Big).
\end{align*}
For any fixed $k\in \mathbb{Z}^+$, there exists $c_k>1$ such
that $\ln^{2k}x\leq x^2, \forall x>c_k$.
Using the fact that ${\rm E}\bar Z_{n,i}^u\to C_i\,
u^{\alpha_{ii}}$ and $n\mb \Phi_n\to\mb \Phi_0$ as $n \to \infty$, 
we obtain that for all sufficiently large $n$,
\begin{align*}
{\rm E}(\ln^{2k}\bar Z_{n,i}^u;\, \bar Z_{n,i}^u>c_k)
\leq {\rm E} (\bar Z_{n,i}^u)^2
=({\rm E}\bar Z_{n,i}^u)^2+\text{Var}(\bar Z_{n,i}^u)
\asymp u^{2\alpha_{ii}}. 
\end{align*}
Therefore, the problem is reduced to proving
${\rm E}(\ln^{2k} \bar Z_{n,i}^u;\,\bar Z_{n,i}^u\leq c_k)
\leq C_0\,n$.
It is sufficient to show
\begin{align*}
{\rm E}\big(\ln^{2k}\bar Z_{n,i}^u;\,\bar Z_{n,i}^u\leq 1\big)
\leq C_0\,n.
\end{align*}
Let $U^2_{{\rm {min}}}={\rm {min}}_{1\leq i \leq n}U_i^2$. Then,
\begin{align}\label{Eq:ZU}
\bar Z_{n,i}^u\stackrel{d}{=}\frac{1}{n}\sum_{j=1}^n \lambda_{j,i}U_j^2
\geq \frac{{\rm trace}(\bb\Lambda_{n,i})}{n}U_{{\rm {min}}}^2
\geq \frac{1}{2} C_i u^{\alpha_{ii}}U_{{\rm {min}}}^2
\triangleq C\,U_{{\rm {min}}}^2.
\end{align}
where the second inequality holds for sufficiently large $n$.

Let $f_n(x)$ be the density function of $U_{{\rm {min}}}^2$, that is $\forall x>0$,
\begin{align*}
f_n(x) =\frac{n}{\sqrt{2\pi x}}\, e^{-x/2}\Big(2\int_{\sqrt x}^\infty \frac{1}{\sqrt{2\pi}}e^{-y^2/2}dy\Big)^{n-1}.
\end{align*}
It is easy to verify that $f_n(x)\leq n/{\sqrt{2\pi x}}$.
It follows from (\ref{Eq:ZU}) that
\begin{align*}
&{\rm E} \big(\ln^{2k}\bar Z_{n,i}^u;\, \bar Z_{n,i}^u\leq 1\big)
\leq {\rm E}\big[\ln^{2k}\big(CU_{{\rm {min}}}^2\big);\,U_{{\rm {min}}}^2 \leq 1/C\big]\\
&=\int_0^{\frac{1}{C}}\ln^{2k}(Cx)f_n(x)dx\leq \frac{n\sqrt{C}}{\sqrt{2\pi}}
\int_{0}^1 y^{-\frac{1}{2}}\ln^{2k}y dy = C_0\,n,
\end{align*}
for all sufficiently large $n$. Therefore, we have proven
\begin{align}
\label{upper bound log Tn}
{\rm E}\big[\ln^{2k}(1+T_{n,i}^u)\big]\leq C_0\,n.
\end{align}
By \eqref{holder ineq}, \eqref{Tn tail upper bound} and
\eqref{upper bound log Tn}, we obtain that when $n$ is large,
\begin{align*}
{\rm E} \big[\big|\ln(1+T_{n,i}^u)\big|^k;\, |T_{n,i}^u|> \xi\big]
\leq C_7 n^{1/2}e^{-C_8\xi\sqrt{n}}\leq C_7 e^{-C_9\xi\sqrt{n}},
\end{align*}
where $C_7$, $C_8$ and $C_9$ are independent of $n$ and $\xi$ and $C_9 < C_8$.
\end{proof}

\section{Appendix}
\label{sec: Appendix}

\begin{proof}[A. Remark on Condition ($\mb A1$)] Let $F_{11}$, $F_{22}$ and $F_{12}$
be the corresponding spectral measures of $C_{11}(\cdot)$, $C_{22}(\cdot)$
and $C_{12}(\cdot)$. By \eqref{covariance assumption} and the Tauberian
Theorem (see, e.g., \cite{Stein_1999Interpo}), we have that as $x \to \infty$,
\begin{align*}
F_{ij}(x,\infty)\sim C_{ij}(0)-C_{ij}(1/x)\sim \widetilde{c}_{ij} |x|^{-\alpha_{ij}},\, i,j=1,2,
\end{align*}
where $ \widetilde{c}_{ii} = c_{ii}$ for $i = 1, 2$ and $ \widetilde{c}_{12}
= c_{12}\rho \sigma_1\sigma_2$.

According to Cramer's theorem (\cite{Chiles_1999}, \cite{Wackernagel_2003}, and \cite{yaglom1987correlation} p.315),
a necessary and sufficient condition for the matrix (\ref{Cov}) to be a valid
covariance function for $\mathbf{X}(t)$ is
\begin{align*}
(F_{12}(B))^2\leq F_{11}(B)F_{22}(B), \ \ \ \forall B\in \mathcal{B}(\R).
\end{align*}
Hence, it is necessary to assume the following conditions on the parameters
$\alpha_{ij}$, $c_{ij}$, $\sigma_i$ ($i = 1, 2$) and $\rho$:
\begin{equation}
\label{restriction on smoothness parameter}
\begin{split}
&\frac{\alpha_{11}+\alpha_{22}}{2} < \alpha_{12}, \ \hbox{ or} \\
&\frac{\alpha_{11}+\alpha_{22}}{2} = \alpha_{12} \ \ \hbox{ and }\  \
c_{12}^2 \rho^2 \sigma_1^2 \sigma_2^2 \le c_{11} c_{22}.
\end{split}
\end{equation}
This shows that (\ref{Eq:valid}) is only slightly stronger than
(\ref{restriction on smoothness parameter}) in the second case, which guarantees that
the bivariate process $\mathbf{X}$ is not degenerate and satisfies (\ref{21}) below. 
\end{proof}

\begin{proof}[B. Proof of \eqref{Xiao95}]
In order to apply Theorem 2.1 in \cite{Xiao_1995} to prove (\ref{Xiao95}), it is sufficient to verify that
there is a constant $c>0$ such that
\begin{equation}\label{21}
{\rm detCov}\big(\mathbf{X}(s)- \mathbf{X}(t)\big) \ge c\, |s-t|^{\alpha_{11} + \alpha_{22}}
\end{equation}
for all $s, t\in [0, 1]$ with sufficiently small $|s-t|$. Here, detCov$(\xi)$ denotes the determinant
of the covariance matrix of the random vector $\xi$. Under Condition ($\mb A1$), we see that for $i = 1, 2$,
\[
\begin{split}
&\E\big[(X_i(s) - X_i(t))^2\big] = 2 C_{ii}(0) - 2C_{ii}(s-t) \sim 2c_{ii} |s-t|^{\alpha_{ii}},\\
&\E\big[(X_1(s) - X_1(t)) (X_2(s) - X_2(t))\big] = 2C_{12} (0)- 2 C_{12}(s-t)\\
& \qquad \qquad \qquad \qquad \qquad \qquad \qquad \quad \
\sim 2 c_{12}\rho \sigma_1\sigma_2 |s-t|^{\alpha_{12}}
\end{split}
\]
as $|s-t| \to 0$. Consequently,
$$ {\rm detCov}\big(\mathbf{X}(s)- \mathbf{X}(t)\big) \sim 4 \; c_{11}c_{22} |s-t|^{\alpha_{11} + \alpha_{22}}
- 4\; c^2_{12}\rho^2 \sigma^2_1\sigma^2_2 |s-t|^{2\alpha_{12}}.
$$
This implies (\ref{21}) and hence proves  (\ref{Xiao95}).
\end{proof}

\begin{proof}[C. Checking the condition (${\mb A2}$) for the bivariate Mat\'{e}rn process] 
Without loss of
generality, assume that $a=1$ and $M_\nu(h):=M(h|\nu,1)$. Denote by $\kappa_\nu=2^{1-\nu}/{\Gamma(\nu)}$, which satisfies $\kappa_{\nu+1}=(2\nu)^{-1}\kappa_\nu $.  Recall that the derivative 
of the Bessel function of the second kind $K_{\nu}$ satisfies the following 
recurrence formula (see, e.g., \cite{Abramowitz_Stegun_1972}, Section $9.6$)
\begin{align*}
K'_{\nu}(z)=-K_{\nu+1}(z)+\frac{\nu}{z}K_{\nu}(z),
\end{align*}
and
for $\ell\in \Z^+\cup \{0\}$, when $\ell<\nu<\ell+1$, we have the following expansion for $M(\cdot)$
\begin{align*}
M_\nu(h)=\sum_{j=0}^{\ell} b_jh^{2j}-b|h|^{2\nu}+o(|t|^{2l+2}),
\end{align*}
where $b_0,\ldots,b_\ell$ are constants and $b=\Gamma(1-\nu)/(2^{2\nu}\Gamma(1+\nu))$ 
(see, e.g., \cite{Stein_1999Interpo}, p. 32).
Hence,
\begin{align*}
&M'_\nu(h)={\rm sgn}(h)(\kappa_\nu \nu |h|^{\nu-1} K_{\nu}(|h|)+\kappa_\nu |h|^\nu K'_{\nu}(|h|))\nonumber\\
&=2\nu\cdot{\rm sgn}(h)|h|^{-1}(M_\nu(h)-M_{\nu+1}(h))\nonumber\\
&=-2\nu b\cdot {\rm sgn}(h)  |h|^{2\nu-1}+o(|h|^{2\nu-1}),
\end{align*}
where $\text{sgn}(h)$ is the sign function. Similarly,
\begin{align*}
M''_\nu(h)&=(2\nu-1){\rm sgn}(h)|h|^{-1}M'_\nu(h)-2\nu\cdot{\rm sgn}(h)|h|^{-1}M'_{\nu+1}(h)\nonumber\\
&=-2\nu(2\nu-1)b\cdot {\rm sgn}^2(h) |h|^{2\nu-2}+o(|h|^{2\nu-2}),\nonumber\\
M^{(3)}_\nu(h)&=(2\nu-2){\rm sgn}(h)|h|^{-1}M''_\nu(h)-2\nu\cdot{\rm sgn}(h)|h|^{-1}M''_{\nu+1}(h),\nonumber\\
&=-2\nu(2\nu-1)(2\nu-2)b\cdot {\rm sgn}^3(h) |h|^{2\nu-3}+o(|h|^{2\nu-3}),\nonumber\\
\cdots&\nonumber\\
M^{(q)}_\nu(h)&=(2\nu-q+1){\rm sgn}(h)|h|^{-1}M^{(q-1)}_\nu(h)-2\nu\cdot{\rm sgn}(h)|h|^{-1}M^{(q-1)}_{\nu+1}(h)\nonumber\\
&=-\frac{b(2\nu)!}{(2\nu-q)!} {\rm sgn}^q(h)  |h|^{2\nu-q}+o(|h|^{2\nu-q}).
\end{align*}
When $q=4$, the nonsmooth bivariate Mat\'{e}rn field $\mb X$
satisfies the regularity condition (${\mb A2}$).
\end{proof}

\noindent\textit{D. Proof of Theorems \ref{thm: asymptotic variance of Z_1&Z_2} 
$\sim$ \ref{thm: asymptotic normality of bar Z}.}
To prove Theorem \ref{thm: asymptotic variance of Z_1&Z_2} and Theorem
 \ref{thm: asymptotic normality of bar Z}, we make use of the following lemma.
\begin{lemma}
\label{lem: asymptotic of Y on h}
If Conditions (${\mb A1}$) and (${\mb A2}$) hold, then as $|h| \to \infty$,
\begin{align}
\label{sigma_n,ii(h)}
\sigma_{n,ii}^{uv}(h)=O(|h|^{\alpha_{ii}-4}),\  \text{uniformly for }\ n>|h|, \, i=1,2 
\end{align}
and
\begin{align}
\label{sigma_n,12(h)}
\sigma_{n,12}^{uv}(h)&=n^{\frac{\alpha_{11}+\alpha_{22}} 2- \alpha_{12}}
O(|h|^{\alpha_{12}-4})
=O(|h|^{\frac{\alpha_{11}+\alpha_{22}}{2}-4})
\end{align}
uniformly for $ n>|h|$.
\end{lemma}
We postpone the proof of Lemma \ref{lem: asymptotic of Y on h} 
to the end of this section.
\begin{proof}[Proof of Theorem \ref{thm: asymptotic variance of Z_1&Z_2}]
Let
\begin{align*}
d_{n,ij}^{uv}(h):=\left\{\begin{array}{ll}
\bigg(1-\frac{|h|}{n}\bigg)(\sigma_{n,ij}^{uv}(h))^2,& |h|<n\\
0,&\text{otherwise}.
\end{array}
\right.
\end{align*}
By \eqref{sigma_n,ii^uv(h)} and \eqref{sigma_n,12^uv(h)}, 
for any fixed $h$, we have $d_{n,ij}^{uv}(h)\rightarrow 
\sigma_{0,ij}^{uv}(h)$ as $n\rightarrow \infty$.
By Lemma \ref{lem: asymptotic of Y on h}, we know
\begin{align*}
d_{n,ij}^{uv}(h)\leq C_0 |h|^{\alpha_{ii}+\alpha_{jj}-8},
\end{align*}
with the power $\alpha_{ii}+\alpha_{jj}-8<-4$. Therefore, 
$\sum_{h=-\infty}^\infty d_{n,ij}^{uv}(h)$ is bounded
by a summable series, and \eqref{nPhi_n convergence} can be 
concluded by the dominated convergence theorem.
\end{proof}
\begin{proof}[Proof of Theorem \ref{thm: asymptotic normality of bar Z}] 
The argument in the following generalizes Kent and Wood \cite{Kent_Wood_1995}'s method 
to the bivariate case. According to the Cram\'{e}r-Wold theorem, it is equivalent 
to prove that for $\forall \bb\gamma =(\gamma_{1,1},\ldots,\gamma_{m,1},\gamma_{1,2},\ldots,\gamma_{m,2})^\top\in 
\mathbb{R}^{2m}$,
\begin{align*}
n^{1/2}\bb\gamma^\top(\mb{\bar Z}_n-{\rm E}[\mb{\bar Z}_n])\xrightarrow{d} \mathcal{N}(0,\bb\gamma^\top\mb\Phi_0\bb\gamma), \ \text{as}\ n\rightarrow \infty.
\end{align*}
Let $\bb\gamma_i:=(\gamma_{1,i},\ldots,\gamma_{m,i})^\top, i=1,2$ and
\begin{align*}
\mb\Gamma_n={\rm diag}(\underbrace{\bb\gamma_1^\top, 
\ldots,
\bb\gamma_1^\top}_{n\ \text{times}},\underbrace{\bb\gamma_2^\top,
\ldots, \bb\gamma_2^\top}_{n\ \text{times}})^\top.
\end{align*}
Therefore, $\bb\Gamma_n$ is a $(2mn)\times (2mn)$ matrix including $n$ 
copies of $\bb\gamma_1$ and $\bb\gamma_2$ on the diagonal.  
Let
\begin{align*}
\mb Y_{n,i}(j):=(Y_{n,i}^1(j),\ldots,Y_{n,i}^m(j))^\top, 
i=1,2,\, j=1,\ldots,n,
\end{align*}
and
\begin{align*}
\mb W_n=(\mb Y_{n,1}^\top(1),\ldots,
\mb Y_{n,1}^\top(n),\mb Y_{n,2}^\top(1),\ldots,
\mb Y_{n,2}^\top(n))^\top.
\end{align*}
Therefore, $\mb W_n$ is a $(2mn)$-dimensional vector. Then, we have
\begin{align*}
S_n\triangleq n^{1/2}\bb\gamma^\top(\mb{\bar Z}_n
-{\rm E}\mb{\bar Z}_n)=n^{-1/2}(\mb W_n^\top\bb\Gamma_n
\mb W_n-{\rm E}(\mb W_n^\top\bb\Gamma_n\mb W_n)).
\end{align*}
Denote by $\mb V_{n}={\rm E}(\mb W_n\mb W_n^\top)$ the covariance 
matrix of $\mb W_n$ and by $\mb V_n^{1/2}$, the Cholesky factor 
of $\mb V_n$, i.e., the lower triangular matrix satisfying
$\mb V_n=\mb V_n^{1/2}(\mb V_n^{1/2})^\top$. Denote 
by $\bb \Lambda_n ={\rm diag}(\lambda_{n,j})_{j=1}^{2mn}$ the 
diagonal matrix whose diagonal entries are eigenvalues
of $2n^{-1/2}(\mb V_n^{1/2})^\top\bb \Gamma_n
\mb V_n^{1/2}$. Then, for a $(2mn)$-dimensional vector 
$\bb\epsilon_n = (\epsilon_{1,n}, \ldots, \epsilon_{2mn,n})^\top$ of 
i.i.d. standard normal random variables, we obtain
\begin{align*}
n^{-\frac{1}{2}}\mb W_n^\top\bb\Gamma_n\mb W_n\stackrel{d}{=}\bb \epsilon_n^\top\big(n^{-\frac{1}{2}}(\mb V_n^{\frac{1}{2}})^\top
\bb \Gamma_n\mb V_n^{\frac{1}{2}}\big) \bb \epsilon_n
\stackrel{d}{=}\frac{1}{2}\bb \epsilon_n^\top\bb \Lambda_n \bb \epsilon_n.
\end{align*}
Therefore, for $\forall \theta<{\rm {min}}_{1\leq j\leq 2mn}\lambda_{n,j}^{-1}$, 
the cumulant generating function $S_n$ is given by
\begin{align*}
k_n(\theta)\triangleq \ln{\rm E}e^{\theta S_n}
=-\frac{1}{2}\sum_{j=1}^{2mn}(\ln(1-\theta \lambda_{n,j})
+\theta \lambda_{n,j})
\end{align*}
To obtain the limit of $k_n(\theta)$ as $n\rightarrow \infty$, 
we first prove
\begin{align}
\label{summation of fourth power of eigenvalues converges}
{\rm trace}(\bb\Lambda_n^4)=\sum_{j=1}^{2mn}\lambda_{n,j}^4\rightarrow 0, 
\ \text{as}\ n\rightarrow \infty.
\end{align}
For $ 1\leq i_1,i_2\leq 2, 1\leq j_1,j_2\leq n, 1\leq k_1,k_2\leq m
$, let
\begin{align}
\ell_1=(i_1-1)mn+(j_1-1)m+k_1,\nonumber\\
\ell_2=(i_2-1)mn+(j_2-1)m+k_2.\nonumber
\end{align}
The $(\ell_1,\ell_2)$ entry of $\mb W_n$ is
\begin{align*}
\mb V_n(\ell_1,\ell_2)={\rm E}[Y_{n,i_1}^{k_1}(j_1)Y_{n,i_2}^{k_2}(j_2)]
=\sigma_{n,i_1i_2}^{k_1k_2}(j_2-j_1).
\end{align*}
Therefore,
\begin{align}
\label{expression_trace Lambda^4}
&{\rm trace}(\bb \Lambda_n^4)=\frac{16}{n^2}{\rm trace}((\mb V_n\bb \Gamma_n)^4)\nonumber\\
&=\frac{16}{n^2}\sum_{\ell_1,\ldots,\ell_4=1}^{2mn}(\mb V_n\bb\Gamma_n)(\ell_1,\ell_2)
(\mb V_n\bb\Gamma_n)(\ell_2,\ell_3)
(\mb V_n\bb\Gamma_n)(\ell_3,\ell_4)(\mb V_n\bb\Gamma_n)(\ell_4,\ell_1)\nonumber\\
&=\frac{16}{n^2}\sum_{i_1,\ldots,i_4=1}^2\sum_{k_1,\ldots,k_4=1}^m
\gamma_{k_1,i_1}\gamma_{k_2,i_2}\gamma_{k_3,i_3}\gamma_{k_4,i_4}
\Delta_n(k_1,\ldots,k_4,i_1,\ldots,i_4),
\end{align}
where
\begin{align*}
&\Delta_n(k_1,\ldots,k_4,i_1,\ldots,i_4)\\
&:=\sum_{j_1,\ldots,j_4=1}^n\sigma_{n,i_1i_2}^{k_1k_2}(j_2-j_1)
\sigma_{n,i_2i_3}^{k_2k_3}(j_3-j_2)\sigma_{n,i_3i_4}^{k_3k_4}(j_4-j_3)
\sigma_{n,i_4i_1}^{k_4k_1}(j_4-j_1).
\end{align*}
Letting $h_i=j_{i+1}-j_{i},i=1,2,3$, we have
\begin{align*}
&\Delta_n(k_1,\ldots,k_4,i_1,\ldots,i_4)\nonumber\\
&=\sum_{j_1,\ldots,j_4=1}^n\sigma_{n,i_1i_2}^{k_1k_2}(h_1)
\sigma_{n,i_2i_3}^{k_2k_3}(h_2)\sigma_{n,i_3i_4}^{k_3k_4}(h_3)
\sigma_{n,i_4i_1}^{k_4k_1}(h_1+h_2+h_3).
\end{align*}
Given fixed $h_1,h_2$ and $h_3$, the cardinality of the set
\begin{align*}
\#\{(j_1,\ldots,j_4)\ |\ 1\leq j_1,\ldots,j_4\leq n\}\leq n.
\end{align*}
Hence,
\begin{align*}
&|\Delta_n(k_1,\ldots,k_4,i_1,\ldots,i_4)|\nonumber\\
\leq & n\sum_{|h_1|,|h_2|,|h_3|\leq n-1} |\sigma_{n,i_1i_2}^{k_1k_2}(h_1)\sigma_{n,i_2i_3}^{k_2k_3}(h_2)
\sigma_{n,i_3i_4}^{k_3k_4}(h_3)\sigma_{n,i_4i_1}^{k_4k_1}(h_1+h_2+h_3)|
\end{align*}
Further, by Lemma \ref{lem: asymptotic of Y on h}, we have
\begin{align}
\label{convergence rate_Delta_n}
&|\Delta_n(k_1,\ldots,k_4,i_1,\ldots,i_4)|\nonumber\\
\leq & C_0 n\prod_{r=1}^3\sum_{h_r=-n+1}^{n-1}h_r^{\frac{\alpha_{i_ri_r}}{2}
+\frac{\alpha_{i_{r+1}i_{r+1}}}{2}-4}\nonumber\\
\leq & C_0 n \prod_{r=1}^3 \sum_{h_r=-\infty}^{\infty}
h_r^{\frac{\alpha_{i_ri_r}}{2}
+\frac{\alpha_{i_{r+1}i_{r+1}}}{2}-4}\nonumber\\
=&O(n).
\end{align}
The last equality holds since $\alpha_{i_ri_r}/2
+\alpha_{i_{r+1}i_{r+1}}/2-4<-2$. By 
\eqref{expression_trace Lambda^4} and \eqref{convergence rate_Delta_n}, 
we have
\begin{align*}
{\rm trace}(\bb\Lambda_n^4)=O(n^{-1})\rightarrow 0, \ \text{as}\ n\rightarrow \infty.
\end{align*}
Now, we are ready to prove the asymptotic normality of $S_n$. 
By applying Taylor's expansion to $\ln(1-\theta\lambda_{n,j})$ at $\theta=0$, 
we obtain
\begin{align*}
k_n(\theta)&=\frac{\theta^2}{4}\sum_{j=1}^{2mn}\lambda_{n,j}^2
+\frac{\theta^3}{6}\sum_{j=1}^{2mn}\lambda_{n,j}^3
+\frac{\theta^4}{8}\sum_{j=1}^{2mn}(1-\theta_{n,j}
\lambda_{n,j})^{-4}\lambda_{n,j}^4,
\end{align*}
where $\theta_{n,j}$ is between $0$ and $\theta$.

Let us first consider the term $\sum_{j=1}^{2mn}\lambda_{n,j}^2/2$. 
Since
\begin{align*}
&\frac{1}{2}\sum_{j=1}^{2mn}\lambda_{n,j}^2=\frac{1}{2}
{\rm trace}(\bb\Lambda_n^2)=\frac{2}{n} 
{\rm trace}((\mb V_n\bb\Gamma_n)^2)\nonumber\\
&=\frac{2}{n} \sum_{i_1,i_2=1}^2\sum_{k_1,k_2=1}^m\sum_{j_1,j_2=1}^n
\gamma_{k_1,i_1}
\gamma_{k_2,i_2}\big(\sigma_{n,i_1i_2}^{k_1k_2}(j_2-j_1)\big)^2,
\end{align*}
and
\begin{align*}
&\gamma^\top\bb\Phi_n\gamma=\sum_{i_1,i_2=1}^2\sum_{k_1,k_2=1}^m\gamma_{k_1,i_1}
\gamma_{k_2,i_2}\phi_{n,i_1i_2}^{k_1k_2}\nonumber\\
=&\frac{2}{n^2}\sum_{i_1,i_2=1}^2\sum_{k_1,k_2=1}^m\sum_{j_1,j_2=1}^n
\gamma_{k_1,i_1}\gamma_{k_2,i_2}\big(\sigma_{n,i_1i_2}^{k_1k_2}(j_2-j_1)\big)^2,
\end{align*}
it follows from Theorem \ref{thm: asymptotic variance of Z_1&Z_2} that
\begin{align}
\label{convergence_sum_lambda^2}
\frac{1}{2}\sum_{j=1}^{2mn}\lambda_{n,j}^2=\gamma^\top(n\bb\Phi_n)\gamma
\rightarrow \gamma^\top \bb\Phi_0\gamma,\ \text{as}\ n\rightarrow \infty.
\end{align}
Secondly, by \eqref{summation of fourth power of eigenvalues converges}, we have
\begin{align}
\label{convergence_max lambda}
{\rm max}_{1\leq j\leq 2mn} |\lambda_{n,j}|\leq \bigg(\sum_{j=1}^{2mn}\lambda_{n,j}^4\bigg)^{\frac{1}{4}}\rightarrow 0,\ \text{as}\ n\rightarrow \infty,
\end{align}
which implies
\begin{align}
\label{convergence_sum_lambda^3}
\bigg|\sum_{j=1}^{2mn} \lambda_{n,j}^3\bigg|\leq {\rm max}_{1\leq j\leq 
2mn}|\lambda_{n,j}|\sum_{j=1}^{2mn}\lambda_{n,j}^2\rightarrow 0,\
\text{as}\ n\rightarrow \infty.
\end{align}
Thirdly, note that $\delta:=\rm sup_{n\geq 1}{\rm max}_{1\leq j\leq 2mn}$ 
$|\lambda_{n,j}|$ is positive and finite by \eqref{convergence_max lambda}. 
If we restrict attention to $|\theta|\leq (2\delta)^{-1}$, 
we have $(1-\theta_{n,j}\lambda_{n,j})^{-4}\leq 16$; hence, for 
$\theta\in (-(2\delta)^{-1},(2\delta)^{-1})$,
\begin{align}
\label{convergence_sum_1-theta_lambda lambda^4}
\sum_{j=1}^{2mn}(1-\theta_{n,j}\lambda_{n,j})^{-4}\lambda_{n,j}^4
\rightarrow 0,\ \text{as}\ n\rightarrow \infty.
\end{align}
Therefore, by \eqref{convergence_sum_lambda^2},\eqref{convergence_sum_lambda^3} 
and \eqref{convergence_sum_1-theta_lambda lambda^4}, for $\forall 
\theta\in (-(2\delta)^{-1},(2\delta)^{-1})$, we have
\begin{align*}
k_n(\theta)\rightarrow \frac{\theta^2}{2}\gamma^\top\bb\Phi_0\gamma,
\end{align*}
which leads to
\begin{align*}
S_n:=n^{1/2}\gamma^\top(\mb{\bar Z}_n-{\rm E}\mb{\bar Z}_n)
\xrightarrow{d} \mathcal{N}(0,\gamma^\top\bb\Phi_0\gamma), 
\ \text{as}\ n\rightarrow \infty.
\end{align*}
This proves Theorem \ref{thm: asymptotic normality of bar Z}.
\end{proof}

Finally, we prove Lemma \ref{lem: asymptotic of Y on h}.
\begin{proof}[Proof of Lemma \ref{lem: asymptotic of Y on h}]
\eqref{sigma_n,ii(h)} comes directly from the proof of Theorem $1$ 
in Kent and Wood \cite{Kent_Wood_1997}. We only need to prove \eqref{sigma_n,12(h)}.
To this end, we expand $C_{12}\big((h+kv-ju)/n\big)$ in a 
Taylor series about $h/n$ to the fourth order to obtain
\begin{align}\label{Eq:L}
\sigma_{n,12}^{uv}(h)&=n^{\frac{\alpha_{11}+\alpha_{22}}{2}}
\sum_{j,k=-1}^1a_ja_kC_{12}\bigg(\frac{h+kv-ju}{n}\bigg)\nonumber\\
&=n^{\frac{\alpha_{11}+\alpha_{22}}{2}}\sum_{r=0}^{3}
\sum_{j,k=-1}^1a_ja_k\frac{(kv-ju)^r}{r!n^r} C_{12}^{(r)}
\bigg(\frac{h}{n}\bigg)\nonumber\\
&\quad+n^{\frac{\alpha_{11}+\alpha_{22}}{2}}\sum_{j,k=-1}^1a_ja_k
\frac{(kv-ju)^{4}}{4!n^{4}} C_{12}^{(4)}\bigg(\frac{h_{kj}^*}{n}\bigg)\nonumber\\
&=n^{\frac{\alpha_{11}+\alpha_{22}}{2}}\sum_{j,k=-1}^1a_ja_k
\frac{(kv-ju)^{4}}{4!n^{4}} C_{12}^{(4)}\bigg(\frac{h_{kj}^*}{n}\bigg),
\end{align}
where $h_{kj}^*$ lies between $h$ and $h+kv-ju$. Since $|kv-ju|
\leq u+v\leq 2m$, $h_{kj}^*\leq 2|h|$ for all $|h|\geq 2m$. 
By applying Condition (${\mb A2}$) to the last terms in (\ref{Eq:L}),
we derive that 
\begin{align*}
|\sigma_{n,12}^{uv}(h)|\leq C_0 |h|^{\alpha_{12}-4}\cdot  n^{\frac{\alpha_{11}+\alpha_{22}}{2}-\alpha_{12}}
\leq C_0 |h|^{\frac{\alpha_{11}+\alpha_{22}}{2}-4}
\end{align*}
for all $|h|\geq 2m$ and all $n>|h|$. This concludes the proof.
\end{proof}

\section{Acknowledgement}
We thank the anonymous reviewers and the associate editor for their thoughtful comments
and helpful suggestions, which have led to several improvements of our manuscript.
\bigskip

\bibliographystyle{plain}
\bibliography{Zhou_Reference}



\end{document}